\documentclass[onefignum,onetabnum]{siamart171218}

\usepackage{amssymb}
\usepackage{amsmath}
\usepackage{cite}
\newsiamremark{remark}{Remark}
\newsiamremark{hypothesis}{Hypothesis}
\crefname{hypothesis}{Hypothesis}{Hypotheses}
\newsiamthm{claim}{Claim}
\usepackage{graphicx}
\usepackage{subcaption}
\usepackage{tikz}
\usepackage{pgfplots}
\usepackage{enumitem, hyperref}

\numberwithin{equation}{section}
\newcommand{\lipsymb}{\mathsf{L}}

\newcommand{\inclu}[0] {\ar@{^{(}->}}

\newcommand{\dist}{{\rm dist}}
\newcommand{\R}{\mathbb{R}}

\newcommand{\EE}{\mathbb{E}}
\newcommand{\trace}{\mathrm{Tr}}
\newcommand{\epi}{\mathrm{epi}\,}
\newcommand{\sign}{\mathrm{sign}}

\newcommand{\RR}{\mathbb{R}}
\newcommand{\cF}{\mathcal{F}}

\newcommand{\proj}{\mathrm{proj}}
\newcommand{\cX}{\mathcal{X}}

\newcommand{\prox}{{\rm prox}}
\newcommand{\dom}{{\rm dom}\,}
\newcommand{\argmin}{\operatornamewithlimits{argmin}}
\newcommand{\ls}{\operatornamewithlimits{limsup}}

\def\cut#1{{}}

\newtheorem{thm}{Theorem}[section]
\newtheorem{lem}[thm]{Lemma}
\newtheorem{cor}[thm]{Corollary}

\newtheorem{assumption}{Assumption}

\newtheorem{example}{Example}[section]

\usepackage{mathtools}
\DeclarePairedDelimiter{\dotp}{\langle}{\rangle}

\headers{Stochastic model-based minimization of weakly convex functions}{Damek Davis and Dmitriy Drusvyatskiy}

\title{Stochastic model-based minimization of weakly convex functions\thanks{Submitted to the editors March 29, 2018. The results in this paper are a combination of the two arXiv preprints \cite{davis2018stochastic,model_based_DDDD}.
		\funding{Research of Drusvyatskiy was supported by the AFOSR YIP award FA9550-15-1-0237 and by the NSF DMS   1651851 and CCF 1740551 awards.}}}

\author{Damek Davis\thanks{School of Operations Research and Information Engineering, Cornell University,
		Ithaca, NY 14850, USA 
		(\email{dsd95@cornell.edu}, \url{people.orie.cornell.edu/dsd95}).}
	\and Dmitriy Drusvyatskiy\thanks{Department of Mathematics, University of Washington, 
		Seattle, WA 98195 
		(\email{ddrusv@uw.edu}, \url{www.math.washington.edu/\~ddrusv}).}
	}

\ifpdf
\hypersetup{
  pdftitle={Stochastic model-based minimization of weakly convex functions},
  pdfauthor={Damek Davis and Dmitriy Drusvyatskiy}
}
\fi

\begin{document}

\maketitle

\begin{abstract}
	We consider a family of algorithms that successively sample and minimize  simple stochastic models of the objective function.  
We show that under reasonable conditions on approximation quality and regularity of the models, any such algorithm drives a natural stationarity measure to zero at the rate $O(k^{-1/4})$. 
	As a consequence, we obtain the first complexity guarantees for the stochastic proximal point, proximal subgradient, and regularized Gauss-Newton methods for minimizing compositions of convex functions with smooth maps. The guiding principle, underlying the complexity guarantees, is that all algorithms under consideration can be interpreted as approximate descent methods on an implicit smoothing of the problem, given by the Moreau envelope. Specializing to classical circumstances, we obtain the long-sought convergence rate of the stochastic projected gradient method, without batching, for minimizing a smooth function on a closed convex set. 
\end{abstract}



\section{Introduction}
Stochastic optimization plays a central role in the statistical sciences, underlying all aspects of learning from data.  The goal of stochastic optimization in data science is to learn a decision rule from a limited data sample, which generalizes well to the entire population. Learning such a decision rule amounts to  minimizing the {\em regularized population risk}: 
\begin{align}\label{eqn:SO}
\min_{x \in \RR^d}~ \varphi(x)=f(x)+r(x)\qquad \textrm{where}\qquad f(x)=\EE_{\xi\sim P}[f(x,\xi)].\tag{{SO}}
\end{align}
Here, $\xi$ encodes the population data, which is assumed to follow some fixed but unknown probability distribution $P$.  The functions $f$ and $r$ play qualitatively different roles. Typically, $f(x,\xi)$ evaluates the loss of the decision rule parametrized by $x$ on a data point $\xi$.  
In contrast, the function $r\colon\R^d\to\R\cup\{+\infty\}$ models constraints on the parameters $x$ or encourages $x$ to have some low dimensional structure, such as sparsity or low rank. Within a Bayesian framework, the regularizer $r$ can model  prior distributional information on $x$.

Robbins-Monro's pioneering 1951 work \cite{MR0042668} gave the first procedure for solving (\ref{eqn:SO}) in the setting when $f(\cdot, \xi)$ are smooth and strongly convex and $r=0$, thereby inspiring decades of further research. Among such algorithms, the proximal stochastic (sub)gradient method is the most successful and widely used in practice. This method constructs a sequence of approximations $x_t$ of the minimizer of (\ref{eqn:SO}) by iterating:
\begin{equation}\label{eqn:SG}\left\{
	\begin{aligned}
	&\textrm{Sample } \xi_t \sim P \\
	& \textrm{Set } x_{t+1}= \prox_{\alpha_t r}\left(x_t - \alpha_t \nabla_x f(x_t, \xi_t)\right)
	\end{aligned}\right\},\tag{{SG}}
	\end{equation}
where $\alpha_t>0$ is an appropriate control sequence and $\prox_{\alpha r}(\cdot)$ is the proximal map
 \begin{equation*}
\prox_{\alpha r}(x):=\argmin_{y}\, \left\{r(y)+\tfrac{1}{2\alpha}\|y-x\|^2\right\}.
\end{equation*}
Thus, in each iteration, the method travels from $x_t$ in the direction opposite to a sampled gradient $\nabla_x f(x_t, \xi_t)$, followed by a proximal operation.

Nonsmooth convex problems may be similarly optimized by replacing sample gradients by sample subgradients $v_t\in \partial_x f(x_t,\xi_t)$, where $\partial_x f(x, \xi)$ is the subdifferential in the sense of convex analysis \cite{con_ter}. Even more broadly, when $f(\cdot, \xi)$ is neither smooth nor convex, the symbol $\partial_x f(\cdot, \xi)$ may refer to a generalized subdifferential. The formal definition of the subdifferential appears in Section~\ref{sec:main}, and is standard in the optimization literature (e.g. \cite[Definition 8.3]{RW98}). 
The reader should keep in mind that in practice, the functions $f(\cdot,\xi_t)$ are often all differentiable along the iterate sequence $\{x_t\}$. Therefore from the viewpoint of implementation, one always computes the true gradients of the sampled functions, using conventional means.  On the other hand, the nonsmoothness cannot be ignored in the  analysis, since $(i)$ the gradients do not vary continuously and  $(ii)$ the objective function is can be nonsmooth at every limit point of the process. We will expand on these two observations shortly.

Performance of stochastic optimization methods is best judged by their \emph{sample complexity} -- the number of i.i.d. realizations $\xi_1, \ldots, \xi_N \sim P$ needed to reach a desired accuracy of the decision rule. Classical results~\cite{complexity} stipulate that for convex problems, it suffices to generate $O(\varepsilon^{-2})$ samples to reach functional accuracy $\varepsilon$ in expectation, and this complexity is unimprovable without making stronger assumptions. For smooth problems, the stochastic gradient method has sample complexity of  $O(\varepsilon^{-4})$ to reach a point with the gradient norm at most $\varepsilon$ in expectation \cite{ghad,Ghadimi2016mini,wotao}.

Despite the ubiquity of the stochastic subgradient method in applications, its sample complexity is not yet known for any reasonably wide class of problems beyond those that are smooth or convex.
This is somewhat concerning as the stochastic subgradient method is the simplest and most widely-used optimization algorithm for large-scale problems arising in machine learning and is the core optimization subroutine in industry backed software libraries, such as Google's TensorFlow~\cite{tensorflow2015-whitepaper}.

The purpose of this work is to provide the first sample complexity bounds for a number of popular stochastic algorithms on a reasonably broad class of nonsmooth and nonconvex  optimization problems. The problem class we consider captures a variety of important computational tasks in data science, as we illustrate below, while the algorithms we analyze include the proximal stochastic subgradient, proximal point, and regularized Gauss-Newton methods. Before stating the complexity guarantees, we must first explain the ``stationarity measure'' that we  will use  to judge the quality of the iterates. It is this stationarity measure that tends to zero at a controlled rate.

\subsection*{The search for stationary points}
Convex optimization algorithms are judged by the rate at which they decrease the function value along the iterate sequence. Analysis of smooth optimization algorithms  focuses instead on the magnitude of the gradients along the iterates. The situation becomes quite different for problems that are neither smooth nor convex.

As in the smooth setting, the primary goal of nonsmooth nonconvex optimization is the search for stationary points. 
A point $x\in\R^d$ is called {\em stationary} for the problem (\ref{eqn:SO}) if the inclusion $0\in \partial \varphi(x)$ holds. In ``primal terms'', these are precisely the points where the directional derivative of $\varphi$ is nonnegative in every direction. Indeed, under mild conditions on $\varphi$, equality holds \cite[Proposition 8.32]{RW98}:
\begin{equation*} 
\dist(0;\partial \varphi(x))=-\inf_{v:\, \|v\|\leq 1} \varphi'(x;v).
\end{equation*}
Thus a point $x$, satisfying $\dist(0;\partial \varphi(x))\leq \varepsilon$, approximately satisfies first-order necessary conditions for optimality.

An immediate difficulty in analyzing stochastic methods for nonsmooth and nonconvex problems is that it is not a priori clear how to measure the progress of the algorithm. Neither the functional suboptimality gap, $\varphi(x_t)-\min \varphi$, nor the stationarity measure, $\dist(0;\partial \varphi(x_t))$, necessarily tend to zero along the iterate sequence. This difficulty persists even in the simplest setting of minimizing a smooth function on a closed convex set by the stochastic projected gradient method. Indeed, what is missing is a continuous measure of stationarity to monitor, instead of the highly discontinuous function $x\mapsto\dist(0;\partial \varphi(x))$.

\subsection*{Weak convexity and the Moreau envelope}
 In this work, we focus on a class of problems that naturally admit a continuous measure of stationarity. We say that a function $g\colon\R^d\to\R$  is
{\em $\rho$-weakly convex} if the assignment $x\mapsto g(x)+\frac{\rho}{2}\|x\|^2$ is convex. The class of weakly convex functions, first introduced in English in~\cite{Nurminskii1973}, is  broad. It includes all convex functions and smooth functions with Lipschitz continuous gradient.
 More generally,  any function of the form $$g(x) = h(c(x)),$$ with $h$ convex and Lipschitz and $c$ a smooth map with Lipschitz Jacobian, is weakly convex \cite[Lemma 4.2]{comp_DP}. Notice that such composite functions need not be smooth nor convex; instead, the composite function class nicely interpolates between the smooth and convex settings. Classical literature highlights the importance of weak convexity in optimization \cite{fav_C2,amen,prox_reg}, while recent advances in statistical learning and signal processing have further reinvigorated the problem class. Nonlinear least squares, phase retrieval  \cite{eM,duchi_ruan_PR, proj_weak_dim}, minimization of the Conditional Value-at-Risk \cite{bental_teb_86,MR2332265,Rockafellar00optimizationof}, graph synchronization  \cite{ban_boum,ang_sing,abbe_band}, covariance estimation \cite{MR3367819}, and robust principal component analysis \cite{rob_cand,chand} naturally lead to weakly convex formulations. For a recent discussion on the role of weak convexity in large-scale optimization, see e.g., \cite{prox_point_surv}. 
 
 It has been known since Nurminskii's work \cite{Nurminskii1974,Nurminskii1973} that when the functions $f(\cdot,\xi)$ are $\rho$-weakly convex and $r=0$, the stochastic subgradient method on \eqref{eqn:SO} generates an iterate sequence that subsequentially converges to a stationary point of the problem, almost surely. Nonetheless, the sample complexity of the basic method and of its proximal extension, has remained elusive. Our approach to resolving this open question relies on an elementary observation: weakly convex problems naturally admit a continuous measure of stationarity through implicit smoothing.
The key construction we use is the {\em Moreau envelope} \cite{MR0201952}:
$$\varphi_{\lambda}(x):=\min_{y}\, \left\{\varphi(y)+\tfrac{1}{2\lambda}\|y-x\|^2\right\},$$
where $\lambda > 0$. Standard results (e.g. \cite{MR0201952}, \cite[Theorem 31.5]{rock}) show that as long as $\varphi$ is $\rho$-weakly convex and $\lambda<\rho^{-1}$, the envelope $\varphi_{\lambda}$ is $C^1$-smooth with the gradient  given by 
\begin{equation*}
\nabla \varphi_{\lambda}(x)=\lambda^{-1}(x-\prox_{\lambda \varphi}(x)).
\end{equation*}
See Figure~\ref{fig:morea_env} for an illustration.

\begin{figure}[h!]
\begin{subfigure}{.5\textwidth}
  \centering
  \includegraphics[scale=0.3]{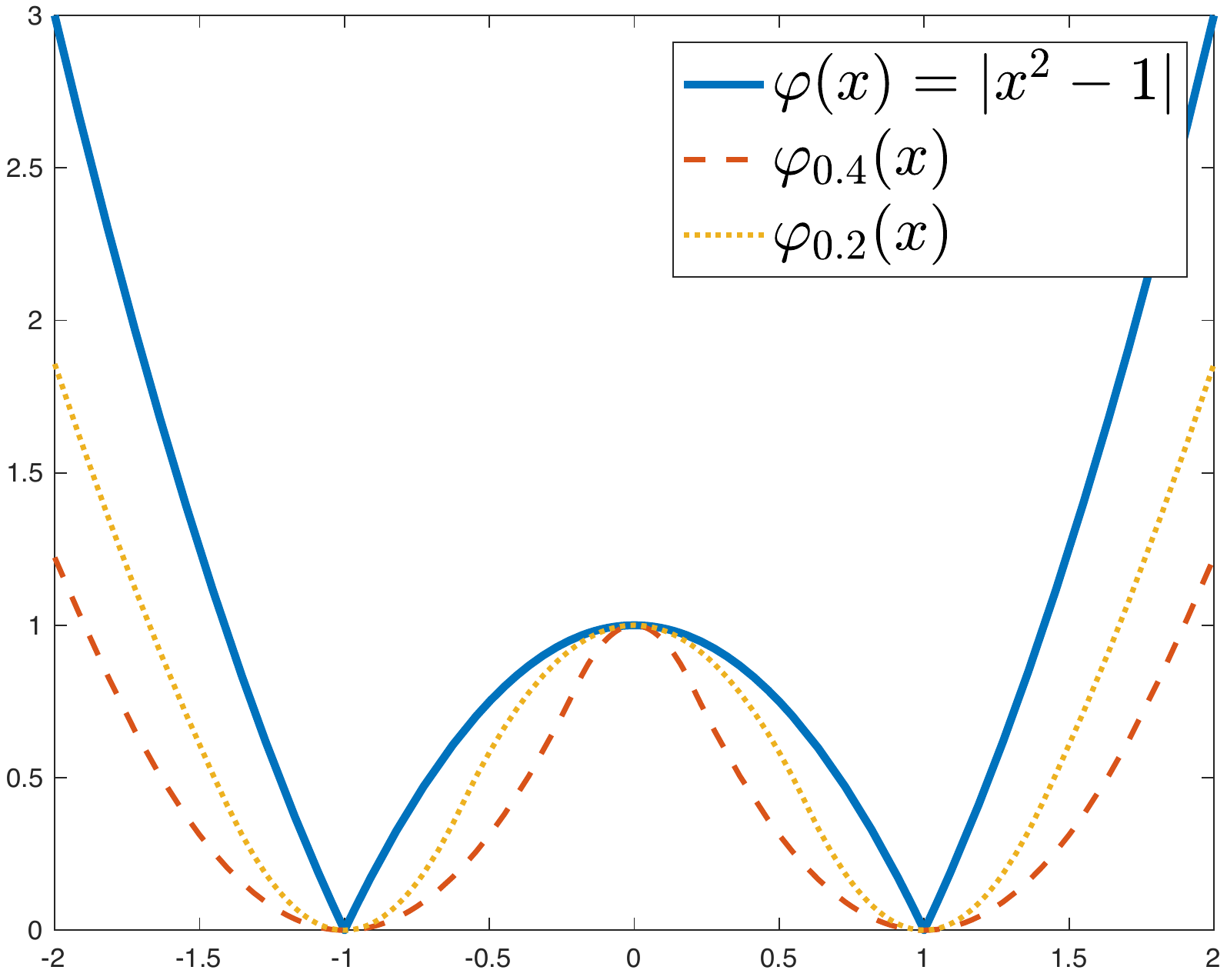}
  \caption{Moreau envelope of $\varphi(x)=|x^2-1|$}
  \label{fig:morea_env}
\end{subfigure}%
\begin{subfigure}{.5\textwidth}
  \centering
  \includegraphics[scale=0.68]{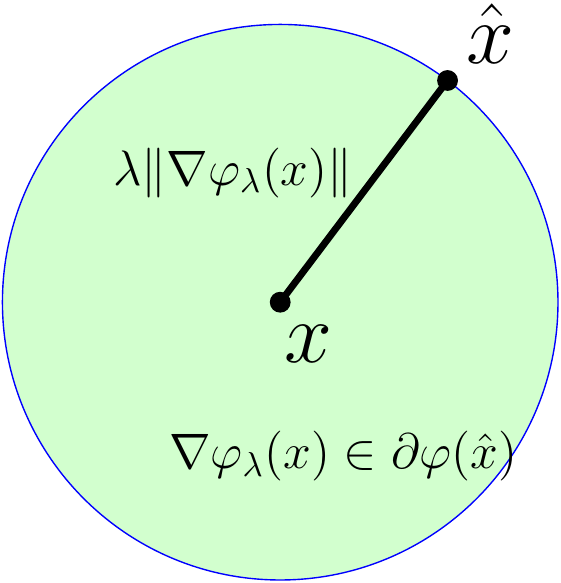}
  \caption{Approximate stationarity}
  \label{fig:approx_stat}
\end{subfigure}
  \caption{An illustration of the Moreau envelope}
\end{figure}

When $f$ is $C^1$-smooth with $\beta$-Lipschitz gradient and there is no regularizer $r$, the norm $\|\nabla \varphi_{1/\beta}(x)\|$ is proportional to the magnitude of the true gradient $\|\nabla f(x)\|$. More generally, when $f$ is $C^1$-smooth and $r$ is nonzero, the norm $\|\nabla \varphi_{1/\beta}(x)\|$ is proportional to the size of the proximal gradient step, commonly used to measure convergence in additive composite minimization \cite{nest_conv_comp}. See the end of Section~\ref{sec:subdiff_mor_env} for a precise statement.
In the broader nonsmooth setting, the norm of the gradient $\|\nabla \varphi_{\lambda}(x)\|$ has an intuitive interpretation in terms of near-stationarity for the target problem $\min_{x} \varphi(x)$. Namely, the definition of the Moreau envelope directly implies that for any point $x\in\R^d$, the proximal point $\hat x:=\prox_{\lambda \varphi}(x)$ satisfies
\begin{equation*}
\left\{\begin{array}{cl}
\|\hat{x}-x\|&=  \lambda\|\nabla \varphi_{\lambda}(x)\|,\\ 
\varphi(\hat x) &\leq \varphi(x),\\
\dist(0;\partial \varphi(\hat{x}))&\leq \|\nabla \varphi_{\lambda}(x)\|.
\end{array}\right. 
\end{equation*}
Thus a small gradient $\|\nabla \varphi_{\lambda}(x)\|$ implies that $x$ is {\em near} some point $\hat x$ that is {\em nearly stationary} for $\varphi$; see Figure~\ref{fig:approx_stat}. In the language of numerical analysis, one can interpret algorithms that drive the gradient of the Moreau envelope to zero as being ``backward-stable''.
For a longer discussion of the near-stationarity concept, we refer to reader to    \cite{prox_point_surv} or \cite[Section 4.1]{comp_DP}.

\subsection*{Contributions}
In this paper, we show that as long as the functions $f(\cdot,\xi)+r(\cdot)$ are $\rho$-weakly convex and mild Lipschitz conditions hold, the proximal stochastic subgradient method will generate a point $x$ satisfying $\EE\|\nabla \varphi_{1/(2\rho)}(x)\|\leq \varepsilon$ after at most $O(\varepsilon^{-4})$ iterations. This is perhaps surprising, since neither the Moreau envelope nor the proximal map of $\varphi$ explicitly appear in the definition of the stochastic proximal subgradient method. This work appears to be the first to recognize the Moreau envelope as a useful potential function for analyzing subgradient methods.

Indeed, we will show that the worst-case complexity $O(\varepsilon^{-4})$ holds for a much wider family of algorithms than the stochastic subgradient method. Setting the stage, recall that the stochastic subgradient method relies on sampling subgradient estimates of $f$, or equivalently sampling good linear models of the function. More broadly,  suppose that $f$ is an arbitrary function (not necessarily written as an expectation), and for every point $x$ we have available a family of ``models'' $\{f_x(\cdot,\xi)\}_{\xi\sim P}$, indexed by a random element $\xi\sim P$. The oracle concept we use assumes that the only access to $f$ is by sampling a model $f_x(\cdot,\xi)$ centered around any base point $x$. Naturally, to make use of such models we must have some control on their approximation quality. We will call the assignment $(x,y,\xi)\mapsto f_x(y,\xi)$ a {\em stochastic one-sided model} if it satisfies
\begin{equation}\label{eqn:err_approx_intro}
\mathbb{E}_{\xi}[f_{x}(x,\xi)]=f(x)\qquad \textrm{and}\qquad \mathbb{E}_{\xi}[f_{x}(y,\xi)-f(y)]\leq \frac{\tau}{2}\|y-x\|^2_2\qquad\forall x,y,
\end{equation}
Thus in each expectation, each model $f_{x}(\cdot,\xi)$ should lower bound $f$ up to a quadratic error, while agreeing with $f$ at the basepoint $x$. See Figure~\ref{fig:illustr_lower_model} for an illustration.

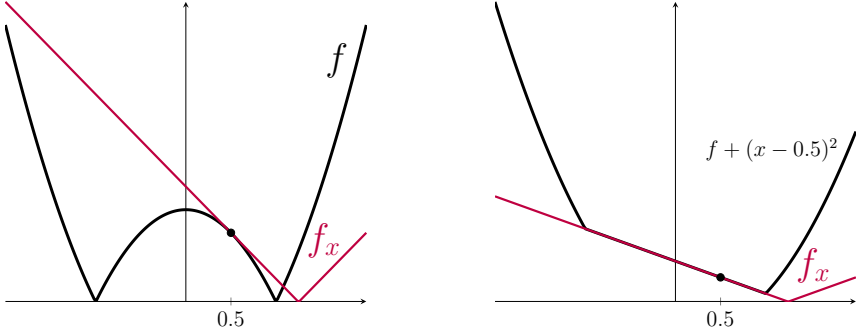
\begin{figure}[h!]
\begin{subfigure}{.5\textwidth}
	\begin{center}
		\begin{tikzpicture}[scale=0.7]
		\pgfplotsset{every tick label/.append style={font=\large}}
		\begin{axis}[%
		domain = -2:2,
		samples = 200,
		axis x line = center,
		axis y line = center,
		xtick={0.5},
		ytick=\empty
		]		
		\addplot[black, ultra thick] {abs(x^2-1)} [yshift=3pt] node[pos=.95,left] {\huge{$f$}};
		 {         \addplot[purple, very thick] {abs(0.75-(x-0.5))} [yshift=6pt] node[pos=.95,left] {\huge{$f_x$}};      }
		
		\addplot [only marks,mark=*] coordinates { (0.5,0.75) };
		\end{axis}		
		\end{tikzpicture}
	\end{center}
\end{subfigure}%
\begin{subfigure}{.5\textwidth}
	\begin{center}
		\begin{tikzpicture}[scale=0.7]
		\pgfplotsset{every tick label/.append style={font=\large}}
		\begin{axis}[%
		domain = -2:2,
		samples = 200,
		axis x line = center,
		axis y line = center,
		xtick={0.5},
		ytick=\empty
		]		
		\addplot[black, ultra thick] {abs(x^2-1)+(x-0.5)^2} [yshift=3pt] node[pos=.95,left] {\large{$f+(x-0.5)^2$}};      
		 {         \addplot[purple, very thick] {abs(0.75-(x-0.5))} [yshift=10pt] node[pos=.95,left] {\huge{$f_x$}};      }
		
		\addplot [only marks,mark=*] coordinates { (0.5,0.75) };
		\end{axis}		
		\end{tikzpicture}
	\end{center}
\end{subfigure}
  \caption{Illustration of a deterministic lower model: $f(x)=|x^2-1|$, $f_{0.5}(y)=|1.25-y|$}
  \label{fig:illustr_lower_model}
\end{figure}

The methods we consider then simply iterate the steps:
\begin{equation}\label{eqn:main_alg}
\begin{aligned}
	&\textrm{Sample }\xi_t\sim P,\\
	&\textrm{Set } x_{t+1}=\argmin_y~ \left\{f_{x_t}(y, \xi_t)+r(y)+\frac{1}{2\alpha_t}\|y-x_t\|^2\right\}.
	\end{aligned}
\end{equation}
We will prove that under mild Lipschitz conditions and provided that each function $f_{x}(\cdot,\xi)+r(\cdot)$ is $\rho$-weakly convex, Algorithm \ref{eqn:main_alg} 
finds a point $x$ with $\EE\|\nabla \varphi_{1/2\rho}(x)\|\leq \varepsilon$ after at most $O(\varepsilon^{-4})$ iterations. 
The main principle underlying the convergence guarantees is interesting in its own right. We will show that  Algorithm \ref{eqn:main_alg} can be interpreted as an approximate descent method on the Moreau envelope:
\begin{equation}\label{eqn:moreau_descent}
\EE[\varphi_{\lambda}(x_{t+1})]\leq \EE[\varphi_{\lambda}(x_{t})] -\alpha_t c_1\EE[\|\nabla \varphi_{\lambda}(x_t)\|^2]+\alpha_t^2 c_2,
\end{equation}
where $\lambda,c_1,c_2$ are problem dependent constants.

When the models $f_{x}(\cdot,\xi)$ are true under-estimators of $f$ in expectation, meaning that \eqref{eqn:err_approx_intro} holds with $\tau=0$, and the functions $f_{x}(\cdot,\xi)+r(\cdot)$ are convex, one expects guarantees that are analogous to the stochastic subgradient method for convex minimization. Indeed, we will show that under these circumstances, Algorithm~\eqref{eqn:main_alg} has complexity $O(\varepsilon^{-2})$ in terms of function value. The complexity estimate improves to $O(\frac{1}{\mu \varepsilon})$ when the functions $f_{x}(\cdot,\xi)+r(\cdot)$ are $\mu$-strongly convex. Though the convexity assumption may appear stringent, it does hold in a number of nonclassical circumstances, such as for minimizing the Condition Value-at Risk (cVaR) of a loss function; see  Example~\ref{exa:cVAr} and Section~\ref{sec:alg_examples} for details.

To crystallize the ideas,  consider the setting of stochastic composite optimization, studied recently by Duchi-Ruan \cite{duchi_ruan}: 
$$f(x,\xi)=h\big(c(x,\xi),\xi\big),$$
where the functions $h(\cdot,\xi)$ are convex and the maps $c(\cdot,\xi)$ are smooth. Note that in the simplest setting when $P$ is a discrete distribution on $\{1,\ldots,m\}$, the  problem \eqref{eqn:SO} reduces to minimizing a regularized empirical average of composite functions:
\begin{equation*} 
\min_{x \in \RR^d}~ \varphi(x)=f(x)+r(x)\qquad \textrm{where}\qquad f(x)=\frac{1}{m}\sum_{i=1}^m h_i(c_i(x))
\end{equation*}
In this setting, the following three stochastic one-sided models appear naturally:
\begin{align}
f_{x}(y,\xi)&=f(x)+\langle \nabla c(x,\xi)^Tw(x, \xi), y-x \rangle,\label{eqn:model_subgrad}\\
f_{x}(y,\xi)&=h\big(c(x,\xi)+\nabla c(x,\xi)(y-x),\xi\big),\label{eqn:model_proxlin}\\
f_{x}(y,\xi)&=h\big(c(y,\xi)),\label{eqn:model_proxpoint}
\end{align}
where $w(x, \xi) \in \partial h(c(x,\xi),\xi)$ is a subgradient selection.
Each iteration of Algorithm \ref{eqn:main_alg} with the models \eqref{eqn:model_subgrad} reduces to the {\em stochastic proximal subgradient} update, already mentioned previously. 
When equipped with the models \eqref{eqn:model_proxlin}, the method becomes the {\em stochastic prox-linear algorithm} --- a close variant of Gauss-Newton. Both of these schemes were recently investigated in \cite{duchi_ruan}, where the authors showed that almost surely all limit points are stationary for the problem \eqref{eqn:SO}. Algorithm~\ref{eqn:main_alg} equipped with the models \eqref{eqn:model_proxpoint} is the {\em stochastic proximal-point algorithm}. This scheme was recently considered for convex minimization in~\cite{ryu2014stochastic} and extended to monotone inclusions  in~\cite{bianchi2016ergodic}. Notice that in contrast to the stochastic proximal subgradient method, the stochastic proximal point and prox-linear algorithms require solving an auxiliary subproblem. The advantage of these two schemes is that the models \eqref{eqn:model_proxlin} and \eqref{eqn:model_proxpoint} provide much finer approximation quality, in that they are two-sided instead of one-sided. Indeed, empirical evidence \cite[Section 4]{duchi_ruan} suggests that the latter two algorithms can perform significantly better and are much more robust to the choice of the sequence $\alpha_t$. We also observe this phenomenon in our experiments in Section~\ref{sec:num_ill}.

The outline of the paper is as follows. We begin with Section~\ref{sec:main}, which records some basic notation and results focusing on weak convexity and the Moreau envelope. This section also presents a number of illustrative applications that will be readily amenable to our algorithmic techniques.
We  then present three distinct convergence arguments: for  the stochastic projected subgradient method in Section~\ref{sec:project_subgrad},  for the stochastic proximal subgradient method in Section~\ref{sec:proxmeth}, and  for algorithms based on general stochastic one-sided models in Section~\ref{sec:set_up_env}. Each argument has its own virtue. In particular, our guarantees for the stochastic projected subgradient method place no restriction on the parameters $\alpha_t$ to be used, in contrast to our latter results. The argument for the stochastic proximal subgradient method generalizes verbatim to the setting when $f$ is $C^1$-smooth and the stochastic gradient estimator has bounded variance, instead of a bounded second moment that we assume elsewhere. Section~\ref{sec:set_up_env} applies to the most general classes of algorithms including stochastic proximal subgradient, prox-linear, and proximal point methods.

\subsection*{Context and related literature}
The convergence guarantees we develop for the proximal stochastic subgradient method are new even in simplified cases. Two such settings are $(i)$ when $f(\cdot,\xi)$ are smooth and $r$ is the indicator function of a closed convex set, and $(ii)$ when $f$ is nonsmooth, we have explicit access to the exact subgradients of $f$, and $r = 0$. 

Analogous convergence guarantees when $r$ is an indicator function of a closed convex set were recently established for a different algorithm in \cite{prixm_guide_subgrad}. This scheme proceeds by directly applying the gradient descent method to the Moreau envelope $\varphi_{\lambda}$, with each proximal point $\prox_{\lambda \varphi}(x)$ approximately evaluated by a convex subgradient method. In contrast, we show here that the basic stochastic subgradient method in the fully proximal setting, and without any modification or parameter tuning, already satisfies the desired convergence guarantees. 

Our work  also improves in two fundamental ways on the results in the seminal papers on the stochastic proximal gradient method for smooth functions \cite{ghad,Ghadimi2016mini,wotao}: first, we allow $f(\cdot,\xi)$ to be nonsmooth and second, even when $f(\cdot,\xi)$ are smooth, we do not require the variance of our stochastic estimator for $\nabla f(x_t)$ to decrease as a function of $t$. The second contribution removes the well-known ``mini-batching" requirements common to~\cite{Ghadimi2016mini,wotao}, while the first significantly expands the class of functions for which the rate of convergence of the stochastic proximal subgradient method is known. It is worthwhile to mention that our techniques rely on weak convexity of the regularizer $r$, while~\cite{wotao} makes no such assumption.

The results in this paper are orthogonal to the recent line of work on accelerated rates of convergence for smooth nonconvex finite sum minimization problems, e.g.,~\cite{lei2017non,allen2017natasha,reddi2016proximal,Allenzhu2017-katyusha}. These works crucially exploit the finite sum structure and/or (higher order) smoothness of the objective functions to push beyond the $O(\varepsilon^{-4})$ complexity. We leave it as an intriguing open question whether such improvement is possible for the nonsmooth weakly convex setting we consider here.

The unifying concept of stochastic one-sided models has not been explicitly used before. The complexity guarantees for the proximal stochastic subgradient, prox-linear, and proximal point methods (Theorem~\ref{thm:main_theorem}) for stochastic composite minimization are new and nicely complement the recent paper \cite{duchi_ruan}. There, the authors proved that almost surely all limit points of the first two methods are stationary.
For a historical account of the prox-linear method, see e.g.,  \cite{burke_com, prox, comp_DP} and the references therein.
For a systematic study of two-sided models (e.g. \eqref{eqn:model_proxlin} and \eqref{eqn:model_proxpoint}) in optimization, see \cite{taylor}. Stochastic compositional problems have also appeared in the parallel line of work \cite{wang_mengdi}. There, the authors require the entire composite function to be either convex or smooth. We make no such assumptions here. 

The convergence rate of Algorithm~\ref{eqn:main_alg} in terms of function values in the convex setting is presented in Theorems~\ref{thm:convergenceC_nonstrong}, \ref{thm:convergenceC}, and is intriguing. Even specializing to the proximal stochastic subgradient method,  Theorems~\ref{thm:convergenceC_nonstrong} and \ref{thm:convergenceC} appear to be stronger than the state of the art. Namely, in contrast to previous work \cite{prox_subgrad_duchi,cruz_subgrad},  the norms of the subgradients of $r$ do not enter the complexity bounds established in Theorem~\ref{thm:convergenceC_nonstrong}, while Theorem~\ref{thm:convergenceC}  extends the nonuniform averaging technique of \cite{simpler_subgrad} for strongly convex minimization to the fully proximal setting.

The observation that Algorithm~\ref{eqn:main_alg} is an approximate descent method on the Moreau envelope \eqref{eqn:moreau_descent} is tangentially related to the recent work on ``inexact first-order oracles'' in convex optimization \cite{Nesterov2015,Devolder2014} and its partial extensions to nonconvex settings \cite{Dvurechensky}. Expanding on the precise relationship between the techniques is an intriguing open question.

\section{Basic notation and preliminaries}\label{sec:main}
Throughout, we consider a Euclidean space $\R^d$ endowed with an inner product $\langle\cdot,\cdot \rangle$ and the induced norm $\|x\|=\sqrt{\langle x, x\rangle}$. For any function $\varphi\colon\R^d\to\R\cup\{\infty\}$, the {\em domain} and {\em epigraph} are the sets
\begin{equation*}
\dom \varphi=\{x\in \R^d: \varphi(x)<\infty\},\qquad 
\epi \varphi=\{(x,r)\in \R^d\times\R: r\geq \varphi(x)\},
\end{equation*}
respectively. We say that $\varphi\colon\R^d\to\R\cup\{\infty\}$ is {\em closed} if the $\epi \varphi$ is a closed set.

This work focuses on algorithms for minimizing weakly convex functions.\footnote{To the best of our knowledge, the class of weakly convex functions was introduced in~\cite{Nurminskii1973}.} A function $\varphi\colon\R^d\to\R\cup\{\infty\}$ is called {\em $\rho$-weakly convex} if the assignment $x\mapsto 
\varphi(x)+\frac{\rho}{2}\|x\|^2$ is a convex function. In this section, we summarize some basic properties of this function class. All results we state in this section are either standard, or follow quickly from analogous results for convex functions. For further details and a historical account, we refer the reader to the short note \cite{prox_point_surv}.

\subsection{Examples of weakly convex functions}
Weakly convex functions are widespread in applications and are typically easy to 
recognize. One common source  is the composite function class:
\begin{equation}\label{eqn:comp}
\varphi(x):=h(c(x)),
\end{equation}
where $h\colon\R^m\to\R$ is convex and $L$-Lipschitz and $c\colon\R^d\to\R^m$ is a 
$C^1$-smooth map with $\beta$-Lipschitz continuous Jacobian. An easy argument shows that 
the composite function $\varphi$ is $L\beta$-weakly convex \cite[Lemma 4.2]{comp_DP}. Below, we list a few examples to illustrate how widespread this problem class is in large-scale data scientific applications. The examples are here only to set the context; the reader can safely skip this discussion during the initial reading.

\begin{example}[Robust phase retrieval]\label{exa:phase_retr}
	{\rm
	Phase retrieval is a common computational problem, with applications in 
diverse areas such as imaging, X-ray crystallography, and speech processing. 
For simplicity, we will consider the version of the problem over the reals.
The (real-valued) phase retrieval problem seeks to determine a point $x$ satisfying the 
magnitude conditions, $$|\langle a_i,x\rangle|\approx b_i\quad \textrm{for 
}i=1,\ldots,m,$$ where $a_i\in \R^d$ and $b_i\in\R$ are given. Whenever
gross outliers occur in the measurements $b_i$, the following robust 
formulation 
of the problem is appealing \cite{eM,duchi_ruan_PR, proj_weak_dim}:
$$\min_x ~\frac{1}{m}\sum_{i=1}^m |\langle a_i,x\rangle^2-b_i^2|.$$
The use of the $\ell_1$ penalty promotes strong recovery and stability properties even in the noiseless setting \cite{duchi_ruan_PR,eM}. Numerous 
other  nonconvex approaches to phase retrieval exist, which rely on different 
problem 
formulations; for example, \cite{wirt_flow,rand_quad,phase_nonconv}.}	
\end{example}

\begin{example}[Covariance matrix estimation]{\rm
		The problem of covariance estimation from quadratic measurements, introduced in  \cite{chen2015exact}, is a higher rank variant of  phase retrieval. Let $a_1, \ldots, a_m \in\R^d$ be measurement vectors. The goal is to recover a low rank decomposition of a covariance matrix $\bar X\bar X^T$, with $\bar X\in\RR^{d \times r}$ for a given $0 \le r \le d$, from quadratic measurements
		$$
		b_i  \approx a_i^T \bar X\bar X^T a_i  = \trace(\bar X\bar X^T a_ia_i^T).
		$$
		Note that we can only recover $\bar X$ up to multiplication by an orthogonal matrix.
		This problem arises in a variety of contexts, such as covariance sketching for data streams and spectrum estimation of stochastic processes. We refer the reader to \cite{chen2015exact} for details.
		Supposing that $m$ is even, the authors of \cite{chen2015exact} show that the following potential function has strong recovery guarantees under usual statistical assumptions:  
		\begin{equation}\label{eqn:cov_est}
		\min_{X\in \R^{d\times r}} ~\frac{1}{m}\sum_{i=1}^m \left|\left\langle XX^T,a_{2i}a_{2i}^T-a_{2i-1}a_{2i-1}^T\right\rangle-(b_{2i}-b_{2i-1})\right|.
		\end{equation}
}
\end{example}

\begin{example}[Blind deconvolution and biconvex compressive sensing]{\hfill \\ }
{\rm The problem of blind deconvolution seeks to recover a pair of vectors in two low-dimensional structured spaces from their pairwise convolution. This problem occurs in a number of fields, such as astronomy and computer vision \cite{661187,5963691}. For simplicity focusing on the real-valued case, one appealing formulation of the problem reads
$$\min_{x,y}~ \frac{1}{m}\sum_{i=1}^m | \langle u_i,x\rangle \langle v_i, y\rangle-b_i|,$$
where $u_i$ and $v_i$ are known vectors, and $b_i$ are the convolution measurements. More broadly, problems of this form fall within the area of biconvex compressive sensing \cite{MR3424852}.
Similarly to the previous two examples, the use of the $\ell_1$-penalty on the residuals allows for strong recovery and identifiability properties of the problem under statistical assumptions. Details will appear in a forthcoming paper. }
\end{example}

\begin{example}[Sparse dictionary learning]
{\rm
The problem of sparse dictionary learning seeks to find a sparse representation of the input data as a linear combination of basic atoms, which comprise the ``dictionary''. This technique is routinely used in image and video processing. More formally, given a set of vectors $\{x_1,\ldots, x_m\}\subset\R^d$, we wish to find a matrix $D\in \R^{d\times n}$ and sparse weights $\{r_1,\ldots,r_m\}\subset\R^n$ such that the error $\|x_i-Dr_i\|_2$ is small for all $i$. The following is a robust variant of the standard relaxation of the problem:
\begin{equation}\label{eqn:sparse_dic_learn}
\min_{D\in \R^{d\times n},~r_1\in \R^n} \frac{1}{m}\sum_{i=1}^m \|x_i-Dr_i\|_2+\lambda\|r_i\|_1\qquad \textrm{subject to}\quad \|D_i\|\leq 1~\forall i.
\end{equation}
More precisely, typical formulations use the squared norm $\|\cdot\|_2^2$ instead of the norm $\|\cdot\|_2$; see e.g. \cite{tosic_dic,sparse_dic_learn_many}. When there are outliers in the data (i.e. not all of the data vectors $x_i$ can be sparsely represented), the formulation \eqref{eqn:sparse_dic_learn} may be more appealing.
}
\end{example}

\begin{example}[Robust PCA]
{\rm
In robust principal component analysis, one seeks to identify sparse corruptions 
of a low-rank matrix \cite{rob_cand,chand}. One typical example is image 
deconvolution,  where the low-rank structure  models the background of an image 
while the sparse corruption models the foreground. 
 Formally, given a $m\times n$ matrix $M$, the goal is to find a decomposition 
$M=L+S$, where $L$ is low rank and $S$ is sparse. A common relaxation of the 
problem is
$$\min_{U\in \R^{m\times r},V\in \R^{n\times r}}~ \|UV^T-M\|_1,$$ 
where $r$ is the target rank. As is common, the entrywise $\ell_1$ norm encourages a sparse residual $UV^T-M$.}
\end{example}

\begin{example}[Conditional Value-at-Risk]\label{exa:cVAr}{\em 
As in the introduction, let $f(x,\xi)$ be a loss of a decision rule parametrized by $x$ on a data point $\xi$, where the population data follows a probability distribution $\xi\sim P$. Rather than minimizing the expectation $f(x)=\EE_{\xi\sim P}f(x,\xi)$, one often wishes to minimize the conditional expectation of the random variable $f(x,\cdot)$ over its $\alpha$-tail, for some fixed $\alpha\in (0,1)$. This quantity is called the Conditional Value-at-Risk (cVaR) and it has a distinguished history. In particular, it is well known from the seminal work \cite{Rockafellar00optimizationof}
that minimizing cVaR of the loss function can be formalized as\footnote{We refer the reader to \cite[pp. 44]{ROCKAFELLAR201333} and \cite{MR2332265} for a historical account of the cVaR minimization formula,  and in particular its interpretation as the ``optimized certainty equivalent'' introduced in \cite{bental_teb_86}.} 
$$\min_{\gamma\in \R, x\in \R^d}~ (1-\alpha)\gamma+\EE_{\xi\sim P}[(f(x,\xi)-\gamma)^+],$$
where  we use the notation $r^+=\max\{0,r\}$. If the loss function $f(\cdot,\xi)$ is $\rho$-weakly convex for a.e. $\xi$, then the entire objective function is  $\rho$-weakly convex jointly in $(\gamma,x)$. In particular, this is the case when $f(\cdot,\xi)$ is $C^1$-smooth with Lipschitz gradient, or when the loss is $f(\cdot,\xi)$ is convex for a.e. $\xi$. Notice that the terms inside the expectation $(f(\cdot,\xi)-\gamma)^+$ are always nonsmooth, even if the loss function $f(\cdot,\xi)$ is smooth.
}
\end{example}

\subsection{Subdifferential and the Moreau envelope}\label{sec:subdiff_mor_env}
A key property of convex functions is that any subgradient yields a global affine under-estimator of the function. It is this availability of global under-estimators that enables convergence guarantees for nonsmooth convex optimization. An analogous property is true for weakly convex functions, with the caveat that the subdifferential is meant in a broader variational analytic sense and the affine under-estimators are replaced by concave quadratic under-estimators. We now formalize this observation.

Consider a function $\varphi\colon\R^d\to\R\cup\{\infty\}$ and a point $x\in \R^d$, with $\varphi(x)$ finite. The {\em subdifferential} of $\varphi$ at $x$, denoted $\partial \varphi(x)$, consists of all vectors $v$ satisfying 
\begin{equation*} 
\varphi(y)\geq \varphi( x)+\langle v,y- x\rangle +o(\|y- x\|)\qquad \textrm{ as }y\to  x.	
\end{equation*}
We set $\partial \varphi(x)=\emptyset$ for all $x\notin\dom \varphi$.
When $\varphi$ is $C^1$-smooth, the subdifferential $\partial \varphi(x)$ consists only of the gradient $\{\nabla \varphi(x)\}$, while for convex functions it reduces to the subdifferential in the sense of convex analysis. The following characterization of weak convexity is standard; we provide a short proof for completeness.

\begin{lemma}[Subdifferential characterization]\label{lem:char_subdiff} {\hfill \\ } The following are equivalent for any lower-semicontinuous function $\varphi\colon \R^d \to\R\cup\{\infty\}$.
	\begin{enumerate}
		\item\label{it:1} The function $f$ is $\rho$-weakly convex.
		\item\label{it:approx_secant} The approximate secant inequality holds:
 \begin{equation}\label{eqn:sec_ineq} 
 \varphi(\lambda x + (1-\lambda)y) \leq \lambda \varphi(x) + (1-\lambda)\varphi(y) + \tfrac{\rho\lambda(1-\lambda)}{2}\|x - y\|^2,
\end{equation}
for all $x,y\in\R^d$ and $\lambda\in [0,1]$.
		\item\label{it:2} The subgradient inequality holds: 
		\begin{equation}\label{eqn:subgrad_ineq}
		\varphi(y)\geq \varphi(x)+\langle v,y-x\rangle-\frac{\rho}{2} \|y-x\|^2,\qquad\forall x,y\in\R^d,~ v\in \partial \varphi(x).
		\end{equation}
		\item\label{it:hyp} The subdifferential map is hypomontone:
		\begin{equation*}\label{eqn:stronger_ineq_hypo}
\langle v - w,x - y\rangle \geq -\rho \|x-y\|^2,
\end{equation*}
for all $x,y\in \R^d$, $v\in \partial \varphi(x)$, and $w \in \partial \varphi(y)$.

	\end{enumerate}
			If $\varphi$ is $C^2$-smooth, then the four properties above are all equivalent to 
			\begin{equation*}
			\nabla^2 \varphi(x)\succeq -\rho I\qquad\qquad \forall x\in \R^d.	
			\end{equation*}
			\end{lemma}
	\begin{proof}
	Algebraic manipulation shows that the usual secant inequality on the function $\varphi+\frac{\rho}{2}\|\cdot\|^2$ is precisely the approximate secant inequality \eqref{eqn:sec_ineq} on $\varphi$. Therefore we deduce the equivalence $\ref{it:1}\Leftrightarrow\ref{it:approx_secant}$. Suppose now \ref{it:1} holds and define the function $g(x)=\varphi(x)+\frac{\rho}{2}\|x\|^2$. Note the equality $\partial g(x)=\partial \varphi(x)+\rho x$; see e.g. \cite[Exercise 8.8]{RW98}. Since $g$ is convex, the inequality, $g(y)\geq g(x)+\langle v+\rho x,y-x\rangle$, holds for all $x,y\in \R^d$ and $v\in \partial \varphi(x)$. Algebraic manipulations then immediately imply \eqref{eqn:subgrad_ineq}, and therefore \ref{it:2} holds. The implication $\ref{it:2}\Rightarrow \ref{it:hyp}$ follows by adding to \eqref{eqn:subgrad_ineq} the analogous inequality with $x$ and $y$ interchanged. Finally suppose that \ref{it:hyp} holds. Algebraic manipulations than imply that the subdifferential of $\varphi+\frac{\rho}{2}\|\cdot\|^2$ is a globally monotone map. Applying \cite[Theorem 12.17]{RW98}, we conclude that $\varphi+\frac{\rho}{2}\|\cdot\|^2$ is convex and therefore 
\ref{it:1} holds. Finally the characterization of weak convexity when $\varphi$	is $C^2$-smooth is immediate from the second-order characterization of convexity of the function $\varphi+\frac{\rho}{2}\|\cdot\|^2$.
	\end{proof}

For any function $\varphi\colon\R^d\to\R\cup\{\infty\}$ and $\lambda>0$, the {\em Moreau envelope} and the {\em proximal map} are defined by 
\begin{align}
\varphi_{\lambda}(x)&:=\min_{y}~ \left\{\varphi(y)+\tfrac{1}{2\lambda}\|y-x\|^2\right\},\label{eqn:moreau}\\
\prox_{\lambda \varphi}(x)&:=\argmin_{y}\, \left\{\varphi(y)+\tfrac{1}{2\lambda}\|y-x\|^2\right\},\notag
\end{align}
respectively. Classically, the Moreau envelope of a convex function is $C^1$-smooth for any $\lambda>0$; see \cite{MR0201952}. The same is true for weakly convex functions, provided $\lambda$ is sufficiently small.

\begin{lemma}\label{lem:moreau_grad} Consider a $\rho$-weakly convex function $\varphi\colon\R^d\to \R\cup\{\infty\}$. Then for any  $\lambda\in (0,\rho^{-1})$, the Moreau envelope $\varphi_{\lambda}$ is $C^1$-smooth with gradient given by
$$\nabla \varphi_{\lambda}(x)=\lambda^{-1}(x-\prox_{\lambda \varphi}(x)).$$
\end{lemma}
See Figure~\ref{fig:morea_env} for an illustration.

As mentioned in the introduction, the norm of the gradient $\|\nabla \varphi_{\lambda}(x)\|$ has an intuitive interpretation in terms of near-stationarity. Namely, the optimality conditions for the minimization problem in \eqref{eqn:moreau} directly imply that for any point $x\in\R^d$, the proximal point $\hat x:=\prox_{\lambda \varphi}(x)$ satisfies
\begin{equation*}
\left\{\begin{array}{cl}
\|\hat{x}-x\|&=  \lambda\|\nabla \varphi_{\lambda}(x)\|,\\ 
\varphi(\hat x) &\leq \varphi(x),\\
\dist(0;\partial \varphi(\hat{x}))&\leq \|\nabla \varphi_{\lambda}(x)\|.
\end{array}\right. 
\end{equation*}
Thus a small gradient $\|\nabla \varphi_{\lambda}(x)\|$ implies that $x$ is {\em near} some point $\hat x$ that is {\em nearly stationary} for $\varphi$; see Figure~\ref{fig:approx_stat}. All of the convergence guarantees that we present will be in terms of the quantity $\|\nabla \varphi_{\lambda}(x)\|$.

 It is important to keep in mind that in more classical circumstances, the size of the gradient of the Moreau envelope is proportional to more familiar quantities. To illustrate, consider the optimization problem 
\begin{equation}\label{eqn:add_comp}
\min_{x\in \R^d}~ \varphi(x):=f(x)+r(x)
\end{equation}
where $f\colon\R^d\to\R$ is $C^1$-smooth with $\rho$-Lipschitz gradient and $r\colon\R^d\to\R\cup\{\infty\}$ is closed and convex. Much of the literature \cite{nest_conv_comp,Ghadimi2016mini} focusing on this problem class highlights the role of the {\em prox-gradient mapping}: 
\begin{equation}\label{eqn:prox_grad_map}
\mathcal{G}_{\lambda}(x)=\lambda^{-1}\left(x-\prox_{\lambda r}(x-\lambda\nabla f(x))\right).
\end{equation}
Indeed, complexity estimates are typically stated in terms of the norm of the  prox-gradient mapping $\|\mathcal{G}_{1/\rho}(x)\|$. On the other hand, one can show that the two stationarity measures, $\|\nabla\varphi_{1/2\rho}(x)\|$ and $\|\mathcal{G}_{1/\rho}(x)\|$, are proportional  \cite[Theorem 4.5]{prox_error}:
$$\tfrac{1}{4}\|\nabla \varphi_{1/2\rho}(x)\|\leq\|\mathcal{G}_{1/\rho}(x)\|\leq \tfrac{3}{2}\left(1+\tfrac{1}{\sqrt{2}}\right)\|\nabla \varphi_{1/2\rho}(x)\|\qquad \forall x\in\R^d.$$
Thus when specializing our results to the  setting \eqref{eqn:add_comp}, 
 all of the convergence guarantees can be immediately translated in terms of the prox-gradient mapping.

\section{Proximal stochastic subgradient method}\label{sec:prox_stoc_subgrad}
In this section, we analyze the proximal stochastic subgradient method for weakly convex minimization. Throughout, we consider the optimization problem
\begin{equation}\label{eqn:gen_err}
\min_{x\in \R^d}~\varphi(x)=f(x)+r(x),
\end{equation}
where  $r\colon\R^d\to\R\cup\{+\infty\}$ is a closed convex function and $f\colon\R^d\to\R$ is a $\rho$-weakly convex function. 
We assume that the only access to $f$ is through a stochastic subgradient oracle.  
\begin{assumption}[Stochastic subgradient oracle]{\rm
Fix a probability space $(\Omega, \cF, P)$ and equip $\RR^d$ with the Borel $\sigma$-algebra.	
We make the following three  assumptions: 
\begin{enumerate}
	\item[(A1)]\label{it1} It is possible to generate i.i.d.\ realizations $\xi_1,\xi_2, \ldots \sim P$. 
	\item[(A2)]\label{it2} There is an open set $U$ containing $\dom r$ and a measurable mapping $G \colon U \times \Omega \rightarrow \RR^d$ satisfying  $\EE_{\xi}[G(x,\xi)]\in \partial f(x)$ for all $x\in U$.
	\item[(A3)]\label{it3} There is a real $L \geq 0$ such that the inequality, $\EE_\xi\left[ \|G(x, \xi)\|^2\right] \leq L^2$, holds for all $x \in \dom r$. 
\end{enumerate}}
\end{assumption}

The three assumption (A1), (A2), (A3) are standard in the literature on stochastic subgradient methods:  assumptions (A1) and (A2) are identical to assumptions (A1) and (A2) in~\cite{doi:10.1137/070704277}, while Assumption (A3) is the same as the assumption listed in~\cite[Equation~(2.5)]{doi:10.1137/070704277}. We will investigate the efficiency of the  proximal stochastic subgradient method, described in	Algorithm~\ref{alg:subgradient}.
\smallskip

\begin{algorithm}[H]
	{\bf Input:} $x_0 \in \dom r$, a sequence $\{\alpha_t\}_{t\geq 0}\subset\R_+$, and iteration count $T$
	
	{\bf Step } $t=0,\ldots,T$:\\		
	\begin{equation*}\left\{
	\begin{aligned}
	&\textrm{Sample } \xi_t \sim P\\
	& \textrm{Set } x_{t+1}=\prox_{\alpha_t r}\left(x_{t} - \alpha_t G(x_t, \xi_t)\right)
	\end{aligned}\right\},
	\end{equation*}
	Sample $t^*\in \{0,\ldots,T\}$ according to  
	$\mathbb{P}(t^*=t)=\frac{\alpha_t}{\textstyle \sum_{i=0}^T \alpha_i}.$
	
	{\bf Return} $x_{t^*}$		
	\caption{Proximal stochastic subgradient method
	}
	\label{alg:subgradient}
\end{algorithm}
\smallskip

Henceforth, the symbol $\mathbb{E}_{t}[\cdot]$ will denote the expectation conditioned on all the realizations $\xi_0,\xi_1,\ldots, \xi_{t-1}$.

\subsection{Projected stochastic subgradient method}\label{sec:project_subgrad}\hfill \\
Our analysis of Algorithm~\ref{alg:subgradient} is shorter and more transparent when $r$ is the indicator function of a closed, convex set $\cX$. This is not surprising, since projected subgradient methods are typically much easier to analyze than their proximal extensions (e.g. \cite{prox_subgrad_duchi,cruz_subgrad}). Note that (\ref{eqn:gen_err}) then reduces to the constrained problem
\begin{equation}\label{eqn:target_constr}
\min_{x\in\cX}~ f(x),
\end{equation}
and the proximal map $\prox_{\alpha  r}(\cdot)$ becomes the nearest point projection $\proj_{\cX}(\cdot)$. Thus throughout Section~\ref{sec:project_subgrad}, we suppose that Assumptions (A1), (A2), and (A3) hold and that $r(\cdot)$ is the indicator function of a closed convex set $\cX$. The following is the main result of this section.

\begin{thm}[Stochastic projected subgradient method]\label{thm:stochastic_sub}
	Let $x_{t^*}$ be the point returned by Algorithm~\ref{alg:subgradient}. 
	Then in terms of any constant  $\bar \rho>\rho$, the estimate holds:
	\begin{equation}\label{eqn:desc_1}
	\EE \left[\varphi_{1/\bar\rho}(x_{t+1})\right]\leq \EE[\varphi_{1/\bar\rho}(x_{t})]-\frac{\alpha_t(\bar \rho-\rho)}{\bar \rho} \EE\left[\|\nabla \varphi_{1/\bar \rho}(x_{t})\|^2\right]+ \frac{\alpha_t^2\bar\rho L^2}{2},
	\end{equation}
	and therefore we have
	\begin{equation}\label{eqn:recurse_tower_1}
	\EE \left[\|\nabla \varphi_{1/\bar \rho}(x_{t^*})\|^2\right]\leq \frac{\bar{\rho}}{\bar \rho-\rho}\cdot\frac{(\varphi_{1/\bar \rho}(x_0) - \min \varphi )+ \frac{\bar\rho L^2}{2}\sum_{t=0}^T \alpha_t^2}{\sum_{t=0}^T \alpha_t}.
		\end{equation} 
	In particular, if Algorithm~\ref{alg:subgradient}  uses the constant parameter $\alpha_t=\frac{\gamma}{\sqrt{T+1}}$, for some real $\gamma>0$, then the point $x_{t^*}$ satisfies:
	\begin{equation}\label{eqn:compl_proj}
	\EE\left[\|\nabla \varphi_{1/2\rho }(x_{t^*})\|^2\right]
	\leq 2\cdot\frac{\left(\varphi_{1/2\rho}(x_0) - \min \varphi\right)+ \rho L^2\gamma^2}{\gamma\sqrt{T+1}}.
	\end{equation} 

\end{thm}
\begin{proof}
	Let $x_{t}$ denote the points generates by Algorithm~\ref{alg:subgradient}. For each index $t$, define $v_t := \EE_t[G(x_t, \xi)] \in \partial f(x_t)$ and set $\hat x_t:= \prox_{\varphi/\bar\rho}(x_t)$. We successively deduce
	\begin{align}
	\EE_t \left[\varphi_{1/\bar\rho}(x_{t+1})\right] &\leq \EE_t \left[f(\hat x_t) + \frac{\bar\rho}{2} \| \hat x_t-x_{t+1}\|^2\right] \label{eqn:prox_def} \\
	&= f(\hat x_t) + \frac{\bar\rho}{2}\EE_t \left[\|\proj_{\cX}(x_{t} - \alpha_t G(x_t, \xi_t)) -  \proj_{\cX}(\hat x_t)\|^2\right] \notag\\
	&\leq f(\hat x_t) + \frac{\bar \rho}{2}\EE_t \left[\|(x_{t}  - \hat x_t)- \alpha_t G(x_t, \xi_t)\|^2 \right]\label{eqn:nonexp_proof1}\\
	&\leq f(\hat x_t) + \frac{\bar\rho}{2} \|x_{t} - \hat x_t\|^2 + \bar\rho\alpha_t\EE_t \left[\dotp{\hat x_t-x_t , G(x_t, \xi_t)}\right] + \frac{\alpha_t^2\bar\rho L^2}{2}\notag \\
	&\leq \varphi_{1/\bar\rho}(x_{t})+ \bar\rho \alpha_t \dotp{\hat x_t - x_t , v_t } + \frac{\alpha_t^2\bar\rho L^2}{2} \notag\\
	&\leq \varphi_{1/\bar\rho}(x_{t}) + \bar\rho \alpha_t \left(f(\hat x_t)-f(x_t)+\frac{\rho}{2}\|x_t-\hat x_t\|^2\right) + \frac{\alpha_t^2\bar\rho L^2}{2}, \label{eqn:weak_conv_proof}
	\end{align}
	where \eqref{eqn:prox_def} follows directly from the definition of the proximal map,  \eqref{eqn:nonexp_proof1} uses that the projection $\proj_{\cX}(\cdot)$ is $1$-Lipschitz, and \eqref{eqn:weak_conv_proof} follows from \eqref{eqn:subgrad_ineq}.
	
		Next, observe that the function $x\mapsto f(x)+\frac{\bar \rho}{2}\|x-x_{t}\|^2$ is  strongly convex with parameter $\bar\rho-\rho$, and therefore
	\begin{align*}
	f(x_{t}) - f(\hat x_{t}) - \frac{\rho}{2}\|x_{t} - \hat x_{t}\|^2
	&=
	\left(f(x_{t})+\frac{\bar\rho}{2}\|x_{t} - x_{t}\|^2\right) - \left( f(\hat x_{t}) + \frac{\bar\rho}{2}\|x_{t} - \hat x_{t}\|^2\right) \\
	&\hspace{20pt}+ \frac{\bar \rho - \rho}{2}\|x_{t} - \hat x_{t}\|^2\\
	&\geq (\bar \rho-\rho)\|x_{t} - \hat x_{t}\|^2 = \frac{\bar \rho-\rho}{\bar\rho^2}\|\nabla \varphi_{1/\bar \rho}(x_{t})\|^2,
	\end{align*}
	where the last equality follows from Lemma~\ref{lem:moreau_grad}.
	Thus we deduce
	$$\EE_t \left[\varphi_{1/\bar\rho}(x_{t+1})\right]\leq \varphi_{1/\bar\rho}(x_{t}) -  \frac{\alpha_t (\bar \rho-\rho)}{\bar\rho}\|\nabla \varphi_{1/\bar \rho}(x_{t})\|^2 + \frac{\alpha_t^2\bar\rho L^2}{2}.$$
	Taking expectations of both sides with respect to $\xi_0,\xi_1,\ldots, \xi_{t-1}$, and using the law of total expectation yields the claimed inequality \eqref{eqn:desc_1}

	Unfolding the recursion \eqref{eqn:desc_1} yields: 
	\begin{align*}
	\EE \left[\varphi_{1/\bar\rho}(x_{T+1})\right] &\leq \varphi_{1/\bar\rho}(x_{0}) + \frac{\bar\rho L^2}{2}\sum_{t = 0}^T \alpha_t^2- \frac{\bar\rho-\rho}{\bar \rho} \cdot\sum_{t = 0}^T\alpha_t  \EE\left[\|\nabla \varphi_{1/\bar \rho}(x_{t})\|^2\right].
	\end{align*} 
	Lower-bounding the left-hand side by $\min \varphi$ and rearranging, we obtain the bound: 
	\begin{equation*}
	\begin{aligned}
	\frac{1}{\sum_{t=0}^T \alpha_t}\sum_{t=0}^T \alpha_t &\EE\left[\|\nabla \varphi_{1/\bar \rho}(x_{t})\|^2\right]\leq \frac{\bar{\rho}}{\bar \rho-\rho}\cdot\frac{\varphi_{1/\bar\rho}(x_{0}) - \min \varphi +  \frac{\bar\rho L^2}{2}\sum_{t=0}^T \alpha_t^2}{\bar\rho\sum_{t=0}^T \alpha_t}.
	\end{aligned}
	\end{equation*}
	Notice that the left-hand-side  is precisely  the expectation $\EE \left[\|\nabla \varphi_{1/\bar \rho}(x_{t^*})\|^2\right]$. Thus \eqref{eqn:recurse_tower_1} holds, as claimed. Finally, \eqref{eqn:compl_proj} follows from \eqref{eqn:recurse_tower_1} by setting $\bar \rho=2\rho$ and $\alpha_t=\frac{\gamma}{\sqrt{T+1}}$ for all indices $t=0,1,\ldots,T$.
\end{proof}

Let us translate the estimate~\eqref{eqn:compl_proj} into a complexity bound by minimizing out in $\gamma$. Namely, suppose we have available some real $ \Delta>0$
satisfying $\Delta\geq \varphi_{1/(2\rho)}(x_0) - \min \varphi$. We deduce from \eqref{eqn:compl_proj} the estimate,
$\EE\left[\|\nabla \varphi_{1/(2\rho) }(x_{t^*})\|^2\right]
\leq 2\cdot\frac{\Delta+ \rho L^2\gamma^2}{\gamma\sqrt{T+1}}.$
Minimizing the right-hand side in $\gamma$ yields the choice
$\gamma=\sqrt{\frac{\Delta}{\rho L^2}}$ 
and therefore the guarantee
\begin{equation}\label{eqn:comp_bound_optim_proj}
\EE\left[\|\nabla \varphi_{1/(2\rho) }(x_{t^*})\|^2\right]\leq 4\cdot\sqrt{\frac{\rho \Delta L^2}{T+1}}.
\end{equation}
In particular, suppose that $f$ is $L$-Lipschitz and the diameter of $\cX$ is bounded by some $D>0$. Then we may set $\Delta:=  \min \left\{ \rho D^2, DL\right\}$, where the first term follows from the definition of the Moreau envelope and the second follows from Lipschitz continuity. 
Then the  number of subgradient evaluations required to find a point $x$ satisfying 
$\EE\|\nabla \varphi_{1/(2\rho) }(x)\|\leq \varepsilon$ is at most 
\begin{equation}\label{eqn:comple1}
\left\lceil 16\cdot\frac{(\rho L D)^2\cdot\min\left\{1,\tfrac{L}{\rho D}\right\}}{\varepsilon^4}\right\rceil.
\end{equation}
This complexity in $\varepsilon$ matches the guarantees of the stochastic gradient method for finding an $\varepsilon$-stationary point of a smooth function \cite[Corollary 2.2]{ghad}.

\subsection*{Improved complexity under convexity} It is intriguing to ask if the complexity \eqref{eqn:comple1} can be improved when $f$ is a convex function. The answer, unsurprisingly, is yes. Since $f$ is convex, here and for the rest of the section, we will let the constant $\rho>0$ be arbitrary. As a first attempt, one may follow the observation of Nesterov \cite{nest_optima} for smooth minimization. The idea is that the right-hand-side of the complexity bound \eqref{eqn:comp_bound_optim_proj} depends on the initial gap $\varphi(x_0)-\min \varphi$. We can make this quantity as small as we wish by a separate subgradient method.
Namely, we may simply run a subgradient method for $T$ iterations to decrease the gap $\varphi(x_0)-\min \varphi$ to $\Delta:=LD/\sqrt{T+1}$; see for example \cite[Proposition 5.5]{nem_jud} for  this basic guarantee. Then we run another round of a subgradient method for $T$ iterations using the optimal choice $\gamma:=\sqrt{\frac{\Delta}{\rho L^2}}$. A quick computation shows that the resulting two-round scheme will find a point $x$ satisfying $\EE\|\nabla\varphi_{1/(2\rho)}(x)\|\leq \varepsilon$ after at most $O(1)\cdot\frac{L^2(\rho D)^{2/3}}{\varepsilon^{8/3}}$ iterations.

By following a completely different technique, introduced by Allen-Zhu \cite{makegradsmall_zhu} for smooth stochastic minimization, this complexity can be even further improved to $\widetilde{O}\left(\frac{(L^2+\rho^2D^2)\log^3(\frac{\rho D}{\varepsilon})}{\varepsilon^2}\right)$ by running logarithmically many rounds of the subgradient method on quadratically regularized problems.  Since this procedure and its analysis is somewhat long and is independent of the rest of the material, we have placed it in an independent arXiv technical report \cite{conv_impr_rate_subgrad}.

\subsection{Proximal stochastic subgradient method}\label{sec:proxmeth}
We next move on to convergence guarantees of Algorithm~\ref{alg:subgradient} in full generality. An important consequence we discuss at the end of the section is a convergence guarantee for the stochastic proximal gradient method 
for minimizing a sum of a smooth function and a convex function, where the gradient oracle has bounded variance (instead of bounded second moment). Those not interested in this guarantee can in principle skip to Section~\ref{sec:set_up_env}, which details our most general convergence result for nonsmooth minimization.

Before we proceed, let us note that for any $x\in U$ and $v\in \partial f(x)$, we have  $\|v\|\leq L$.
To see this, observe that  (A2) and (A3) directly imply that whenever $f$ is differentiable at $x\in U$, we have $$\|\nabla f(x)\|^2=\left\|\EE_{\xi} [G(x,\xi)]\right\|^2\leq  \EE_{\xi} [\left\|G(x,\xi)\right\|^2]\leq L^2.$$
Since at any point $x$, the subdifferential $\partial f(x)$ is the convex hull of limits of gradients at nearby points \cite[Theorem 25.6]{rock}, the claim follows. We will use this estimate in the proof of Lemma~\ref{lem:descent}.

We break up the analysis of Algorithm~\ref{alg:subgradient} into two lemmas. Henceforth, fix a real  $\bar \rho>\rho$. Let $x_t$  be the iterates produced by Algorithm~\ref{alg:subgradient} and let $\xi_t\sim P$ be the i.i.d. realizations used. For each index $t$, define $v_t := \EE_t[G(x_t, \xi)] \in \partial f(x_t)$ and set $\hat x_t:= \prox_{\varphi/\bar\rho}(x_t)$. Observe that by the optimality conditions of the proximal map and the subdifferential sum rule \cite[Exercise 10.10]{RW98}, there exists a vector $\hat v_t \in \partial f(\hat x_t)$ satisfying $\bar \rho(x_t - \hat x_t)  \in \partial r(\hat x_t) +\hat  v_t$. The following lemma realizes $\hat x_t$ as a proximal point of $r$.

\begin{lem}\label{lem:prox_ident}
	For each index $t\geq 0$, equality holds:
	$$
	\hat x_t = \prox_{\alpha_t r}\left(\alpha_t \bar \rho x_t - \alpha_t\hat v_t + (1-\alpha_t \bar \rho)\hat x_t\right).$$
\end{lem}
\begin{proof}
	By the definition of $\hat v_t$, we have 
	\begin{align*}
	\alpha_t \bar \rho(x_t - \hat x_t)  \in \alpha_t\partial r(\hat x_t) + \alpha_t\hat v_t
	\iff &\alpha_t \bar \rho x_t - \alpha_t\hat v_t + (1-\alpha_t \bar \rho)\hat x_t  \in \hat x_t+\alpha_t\partial r(\hat x_t) \\
	\iff &\hat x_t = \prox_{\alpha_t r}(\alpha_t \bar \rho x_t - \alpha_t\hat v_t + (1-\alpha_t \bar \rho)\hat x_t),
	\end{align*}
	where the last equivalence follows from the optimality conditions for the proximal subproblem. This completes the proof. 
\end{proof}

The next lemma establishes a crucial descent property for the iterates. 
\begin{lem}\label{lem:descent}
	Suppose $\bar\rho \in (\rho,2\rho]$ and we have $\alpha_t\in (0,1/\bar\rho]$  for all indices $t\geq 0$. Then the inequality holds:
	\begin{align*}
	\EE_t\|x_{t+1} - \hat x_t\|^2 
		&\leq  \|x_t - \hat x_t \|^2  + 4\alpha_t^2 L^2 - 2\alpha_t(\bar  \rho - \rho)\|x_t - \hat x_t\|^2.
	\end{align*}
\end{lem}
\begin{proof}
	Set $\delta := 1-\alpha_t\bar\rho$. We successively deduce
	\begin{align}
	\EE_t\|x_{t+1} - \hat x_t\|^2&=\EE_t\|\prox_{\alpha_t r} (x_t- \alpha_t G(x_t,\xi_t)) -  \prox_{\alpha_t r}(\alpha_t \bar \rho x_t - \alpha_t\hat v_t + \delta\hat x_t)\|^2\notag\\
	&\leq \EE_t\|x_t - \alpha_t G(x_t, \xi_t) - (\alpha_t \bar \rho x_t - \alpha_t\hat v_t + \delta\hat x_t)\|^2\label{eqn:non_expproof_2} \\
	&= \EE_t\|\delta(x_t - \hat x_t) - \alpha_t( G(x_t, \xi_t) - \hat v_t)\|^2\label{eqn:est_need_later}\\
	&= \delta^2\|x_t - \hat x_t \|^2 - 2\delta\alpha_t\EE_t\left[\dotp{x_t - \hat x_t, G(x_t, \xi_t) - \hat v_t}\right] \notag\\
	&\hspace{20pt}+ \alpha_t^2 \EE_t\| G(x_t, \xi_t) - \hat v_t\|^2\notag \\
	&= \delta^2\|x_t - \hat x_t \|^2 - 2\delta\alpha_t\dotp{x_t - \hat x_t, v_t - \hat v_t} +4\alpha_t^2 L^2 \cut{\alpha_t^2 \EE_t\| G(x_t, \xi_t) - \hat v_t\|^2} \notag\\
	&\leq \delta^2\|x_t - \hat x_t \|^2 + 2\delta\alpha_t\rho \|x_t - \hat x_t\|^2 + 4\alpha_t^2 L^2\label{eqn:hypomon}  \\
	&= (1 - (2\alpha_t(\bar  \rho - \rho)  + \alpha_t^2\bar \rho( 2 \rho-\bar \rho)))\|x_t - \hat x_t \|^2  + 4\alpha_t^2 L^2\cut{ - (2\alpha_t(\bar  \rho - \rho)  + \alpha_t^2\bar \rho( 2 \rho-\bar \rho))\|x_t - \hat x_t\|^2},\notag
	\end{align}
	where the first equation follows from Lemma~\ref{lem:prox_ident},  \eqref{eqn:non_expproof_2} uses that $\prox_{\alpha_t r}(\cdot)$ is $1$-Lipschitz~\cite[Proposition 12.19]{RW98}, and 
	\eqref{eqn:hypomon} follows from  \eqref{eqn:stronger_ineq_hypo}. The result now follows from the assumed inequality $\bar \rho \leq 2\rho$. 
\end{proof}

With Lemma~\ref{lem:descent} proved, we can now establish convergence guarantees of Algorithm~\ref{alg:subgradient} in full generality. 

\begin{thm}[Stochastic proximal subgradient method]\label{thm:stochastic_sub2}
	Fix a real $\bar \rho\in (\rho,2\rho]$ and a
	stepsize sequence $\alpha_t\in (0,1/\bar\rho]$. Then the iterates $x_t$ generated by Algorithm~\ref{alg:subgradient} satisfy
	\begin{equation}\label{eqn:desc_2}
	\EE \left[\varphi_{1/\bar\rho}(x_{t+1})\right]\leq \EE[\varphi_{1/\bar\rho}(x_{t})]-\frac{\alpha_t(\bar \rho-\rho)}{\bar \rho} \EE\left[\|\nabla \varphi_{1/\bar \rho}(x_{t})\|^2\right]+ \alpha_t^2\bar\rho L^2,
	\end{equation}
	and the point $x_{t^*}$ returned by Algorithm~\ref{alg:subgradient} satisfies:
	\begin{equation}\label{eqn:recurse_tower_2}
	\EE \left[\|\nabla \varphi_{1/\bar \rho}(x_{t^*})\|^2\right]\leq \frac{\bar{\rho}}{\bar \rho-\rho}\cdot\frac{(\varphi_{1/\bar \rho}(x_0) - \min \varphi) + 2\bar\rho L^2\sum_{t=0}^T \alpha_t^2}{\sum_{t=0}^T \alpha_t}.
	\end{equation} 
	In particular, if Algorithm~\ref{alg:subgradient}  uses the constant parameter $\alpha_t=\frac{\gamma}{\sqrt{T+1}}$, for some real $\gamma\in (0,\frac{1}{2\rho}]$, then the point $x_{t^*}$ satisfies:
	\begin{equation}\label{eqn:compl_proj2}
	\EE\left[\|\nabla \varphi_{1/2\rho }(x_{t^*})\|^2\right]\leq 2\cdot\frac{\left(\varphi_{1/2 \rho}(x_0) - \min \varphi\right)+ 4\rho L^2\gamma^2}{\gamma\sqrt{T+1}}.
	\end{equation} 
\end{thm}
\begin{proof}
	
	We successively observe
	\begin{align*}
	\EE_t \left[\varphi_{1/\bar\rho}(x_{t+1})\right] &\leq \EE_t \left[\varphi(\hat x_t) + \frac{\bar\rho}{2} \|\hat x_t-x_{t+1}\|^2\right]  \\
	&\leq \varphi(\hat x_t) + \frac{\bar \rho}{2}\left[\|x_t - \hat x_t \|^2+4\alpha_t^2L ^2 - 2\alpha_t(\bar  \rho - \rho)\|x_t - \hat x_t\|^2\right]\\
	&= \varphi_{1/\bar{\rho}}( x_t)+\bar \rho\left[2\alpha_t^2L ^2 - \alpha_t(\bar  \rho - \rho)\|x_t - \hat x_t\|^2\right],
	\end{align*}
	where the first inequality follows directly from the definition of the proximal map and the second follows from Lemma~\ref{lem:descent}. Taking expectations  with respect to $\xi_0,\ldots, \xi_{t-1}$ yields the claimed inequality \eqref{eqn:desc_2}. The rest of the proof proceeds as in Theorem~\ref{thm:stochastic_sub}. Namely, unfolding the recursion \eqref{eqn:desc_2} yields: 
	\begin{align*}
	\EE &\left[\varphi_{1/\bar\rho}(x_{T+1})\right] \leq \varphi_{1/\bar\rho}(x_{0}) + 2\bar\rho L^2\sum_{t = 0}^T \alpha_t^2 - \frac{\bar \rho - \rho}{\bar \rho}\EE\sum_{t = 0}^T\alpha_t\|x_t - \hat x_t\|^2.
	\end{align*} 
	Lower-bounding the left-hand side by $\min \varphi$ and rearranging, we obtain the bound: 
	\begin{equation}\label{eqn:main_ineq2}
	\begin{aligned}
	\frac{1}{\sum_{t=0}^T \alpha_t}&\sum_{t=0}^T \alpha_t \EE\left[\|\nabla \varphi_{1/\bar \rho}(x_t)\|^2\right]\leq \frac{\bar{\rho}}{\bar \rho-\rho}\cdot\frac{(\varphi_{1/\bar \rho}(x_0) - \min \varphi) + 2\bar\rho L^2\sum_{t=0}^T \alpha_t^2}{\sum_{t=0}^T \alpha_t}.
	\end{aligned}
	\end{equation}
	Recognizing the left-hand-side as $\EE\left[\|x_{t^*} - \hat x_{t^*}\|^2\right]$ establishes~\eqref{eqn:recurse_tower_2}. Setting $\bar{\rho}=2\rho$ and $\alpha_t=\frac{\gamma}{\sqrt{T+1}}$ in~\eqref{eqn:recurse_tower_2} yields the final guarantee \eqref{eqn:compl_proj2}.
\end{proof}

\subsection*{Proximal stochastic gradient for smooth minimization}
Let us now look at the consequences of our results in the setting when $f$ is $C^1$-smooth with $\rho$-Lipschitz gradient. Note, that then $f$ is automatically $\rho$-weakly convex. 
In this smooth setting, it is common to replace assumption (A3) with the finite variance condition:
\begin{itemize}
	\item[$\overline{(A3)}$]\label{it4} There is a real $\sigma \geq 0$ such that the inequality, $\EE_\xi\left[ \|G(x, \xi) - \nabla f(x)\|^2\right] \leq \sigma^2$, holds for all $x \in \dom r$. 
\end{itemize}

Henceforth, let us therefore assume that $f$ is $C^1$-smooth with $\rho$-Lipschitz gradient, and Assumptions (A1), (A2), and $\overline{\rm (A3)}$ hold.

All of the results in Section~\ref{sec:proxmeth} can be easily modified to apply to this setting. In particular, Lemma~\ref{lem:prox_ident} holds verbatim, while Lemma~\ref{lem:descent} extends as follows.

\begin{lem}\label{lem:descent_smooth}
	Fix a real $\bar \rho> \rho$ and a
	sequence $\alpha_t\in (0, 1/{\bar{\rho}}]$. Then the inequality holds:
	\begin{align*}
	\EE_t\|x_{t+1} - \hat x_t\|^2 
		&\leq  \|x_t - \hat x_t \|^2  + \alpha_t^2 \sigma^2- \alpha_t(\bar \rho - \rho)\|x_t - \hat x_t\|^2.
	\end{align*}
\end{lem}
\begin{proof}
	By the same argument as in Lemma~\ref{lem:descent}, we arrive at \eqref{eqn:est_need_later} with ${\hat v}_t=\nabla f({\hat x}_t)$. Set $\delta := 1-\alpha_t\bar\rho$ and $w_t := G(x_t, \xi_t) - \nabla f(x_t) $. Adding and subtracting $\nabla f(x_t)$, we successively deduce
	\begin{align}
	\EE_t\|x_{t+1} - \hat x_t\|^2
	&\leq \EE_t\|\delta (x_t - \hat x_t) - \alpha_t( G(x_t, \xi_t) -  \nabla f({\hat x}_t))\|^2\notag\\
	&= \EE_t\|\delta (x_t - \hat x_t) - \alpha_t( \nabla f(x_t)- \nabla f(\hat x_t)) - \alpha_t w_t\|^2\notag\\	
	&= \|\delta (x_t - \hat x_t) - \alpha_t( \nabla f(x_t)- \nabla f(\hat x_t)) \|^2 + \alpha_t^2\EE_t\| w_t \|^2\label{eqn:cross_termcancel3}\\	
	&\leq \delta^2\|x_t - \hat x_t \|^2 - 2\delta\alpha_t\dotp{x_t - \hat x_t, \nabla f(x_t) - \nabla f(\hat x_t)} \notag\\
	&\hspace{20pt} + \alpha_t^2\| \nabla f(x_t)- \nabla f(\hat x_t) \|^2 + \alpha_t^2\sigma^2\label{eqn:expandsquare} \\
	&\leq(\delta^2 + 2\delta \alpha_t\rho + \rho^2\alpha_t^2 )\|x_t - \hat x_t \|^2+\alpha_t^2 \sigma^2 \label{eqn:lip_hyp_stuff}\\
	&= \|x_t - \hat x_t \|^2  + \alpha_t^2 \sigma^2 - \alpha_t(\bar \rho - \rho)(2- \alpha_t(\bar \rho - \rho))\|x_t - \hat x_t\|^2,\notag
		\end{align}
	where \eqref{eqn:cross_termcancel3} follows from assumption (A2), namely  $\EE_t G(x_t, \xi_t)=\nabla f(x_t)$,  \eqref{eqn:expandsquare} follows by expanding the square and using assumption $\overline{(A3)}$, and  \eqref{eqn:lip_hyp_stuff} follows from \eqref{eqn:stronger_ineq_hypo} and  Lipschitz continuity of $\nabla f$. The assumption $\bar \rho \geq \rho$ guarantees $2- \alpha_t(\bar \rho - \rho)\geq 1$. The result follows. 
\end{proof}

We can now state the convergence guarantees of the proximal stochastic gradient method. The proof is completely analogous to that of Theorem~\ref{thm:stochastic_sub2}, with Lemma~\ref{lem:descent_smooth} playing the role of Lemma~\ref{lem:descent}.

\begin{cor}[Stochastic prox-gradient method for smooth minimization]\label{cor:conv_guarant} \hfill \\
	Fix a real $\bar \rho> \rho$ and a
	stepsize sequence $\alpha_t\in (0, 1/\bar{\rho}]$.
	Then the iterates $x_t$ generated by Algorithm~\ref{alg:subgradient} satisfy
	\begin{equation}\label{eqn:desc_3}
	\EE \left[\varphi_{1/\bar\rho}(x_{t+1})\right]\leq \EE[\varphi_{1/\bar\rho}(x_{t})]-\frac{\alpha_t(\bar \rho-\rho)}{2\bar \rho} \EE\left[\|\nabla \varphi_{1/\bar \rho}(x_{t})\|^2\right]+ \frac{\alpha_t^2\bar\rho \sigma^2}{2},
	\end{equation}
	and the point $x_{t^*}$ returned by Algorithm~\ref{alg:subgradient} satisfies:
	\begin{equation}\label{eqn:recurse_tower_3}
	\EE \left[\|\nabla \varphi_{1/\bar \rho}(x_{t^*})\|^2\right]\leq \frac{2\bar{\rho}}{\bar \rho-\rho}\cdot\frac{(\varphi_{1/\bar \rho}(x_0) - \min \varphi) + \frac{\bar\rho \sigma^2}{2}\sum_{t=0}^T \alpha_t^2}{\sum_{t=0}^T \alpha_t}.
	\end{equation} 
	In particular, if Algorithm~\ref{alg:subgradient}  uses the constant parameter $\alpha_t=\frac{\gamma}{\sqrt{T+1}}$, for some real $\gamma\in (0,\frac{1}{2\rho}]$, then the point $x_{t^*}$ satisfies:
	\begin{equation*}
	\EE \left[\|\nabla \varphi_{1/(2\rho)}(x_{t^*})\|^2\right]
	\leq 4\cdot\frac{(\varphi_{1/\bar \rho}(x_0) - \min \varphi) + \rho \sigma^2\gamma^2}{\gamma\sqrt{T+1}}.
	\end{equation*} 
\end{cor}

As mentioned at the end of Section~\ref{sec:subdiff_mor_env}, it is immediate to translate the complexity estimate in Corollary~\ref{cor:conv_guarant} to an analogous estimate in terms of the size of the prox-gradient mapping \eqref{eqn:prox_grad_map}, thereby allowing for a direct comparison with previous results.

\section{Stochastic model-based minimization}\label{sec:set_up_env}
In the previous section, we established the complexity of 
$O(\varepsilon^{-4})$ for the stochastic proximal subgradient methods. In this section, we show that the complexity $O(\varepsilon^{-4})$ persists for a much wider class of algorithms, including the stochastic proximal point and prox-linear algorithms. Henceforth, we consider the optimization problem
\begin{equation}\label{eqn:stoch_approx}
\min_{x\in\R^d}~ \varphi(x) := f(x)+r(x),
\end{equation}
where $r\colon\R^d\to \R\cup\{\infty\}$ is a closed function (not necessarily convex) and $f\colon\R^d\to\R$ is locally Lipschitz. We assume that the only access to $f$ is through a {\em stochastic one-sided model}. 

\begin{assumption}[Stochastic one-sided model]
{\rm
Fix a probability space $(\Omega,\mathcal{F},P)$ and equip $\R^d$ with the Borel $\sigma$-algebra. 
We assume that there exist real $\tau,\eta,\lipsymb\in \R$ such that the following four properties hold: 
\begin{enumerate}
	\item[(B1)] {\bf (Sampling)} It is possible to generate i.i.d.\ realizations $\xi_1,\xi_2, \ldots \sim P$.
	\item[(B2)]\label{it:B2} {\bf (One-sided accuracy)} There is an open convex set $U$ containing $\dom r$ and a measurable function $(x,y,\xi)\mapsto g_x(y,\xi)$, defined on $U\times U\times\Omega$, satisfying  $$\EE_{\xi}\left[f_x(x,\xi)\right]=f(x)
	\qquad \forall x\in U,$$ and 
	$$\EE_{\xi}\left[f_x(y,\xi)-f(y)\right]
	\leq\frac{\tau}{2}\|y-x\|^2\qquad \forall x,y\in U.$$
	\item[(B3)] {\bf (Weak-convexity)} The function $f_x(\cdot,\xi)+r(\cdot)$ is $\eta$-weakly convex $\forall x\in U$, a.e. $\xi\in \Omega$.
	\item[(B4)] {\bf (Lipschitz property)} There exists a measurable function $L \colon \Omega\to \R_+$  satisfying 
	$\sqrt{\EE_{\xi}\left[L(\xi)^2\right]} \leq \lipsymb$
	and such that 
	\begin{equation}\label{eqn:lip_mod}
	f_x(x, \xi) - f_x(y, \xi)  \le  L(\xi)\|x-y\|, 
	\end{equation}
	for all $x,y\in U$ and a.e. $\xi\sim P$.
	
%
\end{enumerate}}
\end{assumption}

It will be useful for the reader to keep in mind the following lemma, which shows that the objective function $\varphi$ is itself weakly convex with parameter $\tau+\eta$ and that $f$ is $\lipsymb$-Lipschitz continuous on $U$.
\begin{lem}\label{lem:weak_conv}
	The function $\varphi$ is $(\tau + \eta)$-weakly convex and the inequality holds:
	\begin{equation}\label{eqn:lip_func}
		|f(x)-f(y)|\leq \lipsymb\|x-y\|,\qquad \textrm{for all }x,y\in U.\end{equation}
\end{lem}
\begin{proof}
	Fix arbitrary points $x, y \in \dom r$ and a real $\lambda \in [0, 1]$, and set $\bar x = \lambda x + (1-\lambda) y$. Define the function $f_x(y):=\EE_{\xi}[f_x(y,\xi)]$.
	Taking into account the equivalence of weak convexity with the approximate secant inequality \eqref{eqn:sec_ineq}, we successively deduce 
	\begin{align}
	\varphi(\bar x)&= \EE_{\xi}\left[r(\bar x) + f_{\bar x}(\bar x, \xi) \right]\label{eqn:1}\\
	&\leq \lambda\EE_{\xi}\left[ r(x) + f_{\bar x}(x, \xi) \right] +  (1-\lambda)\EE_{\xi}\left[ r(y) + f_{\bar x}(y, \xi) \right] + \tfrac{\eta\lambda(1-\lambda)}{2}\|x-y\|^2\label{eqn:2}\\
	&= \lambda(r(x) + f_{\bar x}(x))+  (1-\lambda)(r(y) + f_{\bar x}(y) ) + \tfrac{\eta\lambda(1-\lambda)}{2}\|x-y\|^2\notag\\
	&\leq \lambda\varphi(x)+  (1-\lambda)\varphi(y) + \tfrac{\tau(\lambda^2(1-\lambda) + \lambda(1-\lambda)^2)}{2}\|x-y\|^2 +  \tfrac{\eta\lambda(1-\lambda)}{2}\|x-y\|^2\label{eqn:3}\\
	&= \lambda\varphi(x)+  (1-\lambda)\varphi(y) +  \tfrac{(\tau + \eta)\lambda(1-\lambda)}{2}\|x-y\|^2,\notag
	\end{align}
	where \eqref{eqn:1} uses (B2), inequality \eqref{eqn:2} uses (B3), and \eqref{eqn:3} uses (B2).
	Thus $\varphi$ is $(\tau + \eta)$-weakly convex, as claimed.
	
		Next, taking expectations in (B2) and in \eqref{eqn:lip_mod} yields the estimates:
\begin{align*}
	f(x)-f_x(y)&\leq \lipsymb \|x-y\|\qquad \textrm{and}\qquad f_x(y)-f(y)\leq \frac{\tau}{2}\|x-y\|^2.
\end{align*}
Thus for any point $x\in U$, we deduce 
$$\ls_{y\to x}\frac{f(x)-f(y)}{\|x-y\|}\leq\ls_{y\to x} \frac{\lipsymb \|x-y\|+\frac{\tau}{2}\|y-x\|^2}{\|x-y\|}= \lipsymb.$$
In particular, when $f$ is differentiable at $x$, setting $y=x-s\nabla f(x)$ with $s\searrow 0$, we deduce $\|\nabla f(x)\|\leq \lipsymb$. Since $f$ is locally Lipschitz continuous, its Lipschitz constant on $U$ is no greater than $\sup_{y\in U}\{\|\nabla f(y)\|: f \textrm{ is differentiable at }y\}$.\footnote{This follows by combining gradient formula for the Clarke subdifferential \cite[Theorem 8.1]{CLSW} with the mean value theorem \cite[Theorem 2.4]{CLSW}.} We therefore deduce that $f$ is $\lipsymb$-Lipschitz continuous on $U$, as claimed.
\end{proof}

We can now formalize the algorithm we investigate, as Algorithm~\ref{alg:stoc_prox}.  The reader should note that, in contrast to the  previously discussed algorithms, Algorithm~\ref{alg:stoc_prox} employs a nondecreasing stepsize $\beta_t$, which is inversely proportional to $\alpha_t$. This notational choice will  simplify the analysis and complexity guarantees that follow.

\smallskip
\begin{algorithm}[H]
	{\bf Input:} $x_0\in \R^d$,  real $\bar \rho > \tau + \eta$, a sequence $\{\beta_t\}_{t\geq 0} \subseteq (\bar \rho, \infty)$, and iteration count $T$
	{\bf Step } $t=0,\ldots,T$:\\		
	\begin{equation*}\left\{
	\begin{aligned}
	&\textrm{Sample } \xi_t \sim P\\
	& \textrm{Set } x_{t+1} = \argmin_{x}~ \left\{r(x) + f_{x_t}(x,\xi_t) + \frac{\beta_t}{2} \|x - x_t\|^2\right\}
	\end{aligned}\right\},
	\end{equation*}
	Sample $t^*\in \{0,\ldots,T\}$ according to the discrete probability distribution
	$$\mathbb{P}(t^*=t)\propto\frac{\bar \rho - \tau - \eta}{\beta_t-\eta}.$$
	{\bf Return} $x_{t^*}$		
	\caption{Stochastic Model Based Minimization
	}
	\label{alg:stoc_prox}
\end{algorithm}
\smallskip

\subsection{Analysis of the algorithm}
Henceforth, let $\{x_t\}_{t\geq 0}$ be the iterates generated by Algorithm~\ref{alg:stoc_prox} and let $\{\xi_t\}_{t\geq 0}$ be the corresponding samples used. For each index $t\geq 0$, define the proximal point 
$$\hat x_t=\prox_{\varphi/\bar \rho}(x_t).
$$ 
As in Section~\ref{sec:prox_stoc_subgrad}, we  will use the symbol $\mathbb{E}_{t}[\cdot]$ to denote the expectation conditioned on all the realizations $\xi_0,\xi_1,\ldots, \xi_{t-1}$.
The analysis of Algorithm~\ref{alg:stoc_prox} relies on the following lemma, which establishes two descent type properties. Estimate~\eqref{eqn:recurs_1} is in the same spirit as Lemma~\ref{lem:descent} in Section~\ref{sec:prox_stoc_subgrad}. The estimate~\eqref{eqn:recurs_2}, in contrast, will be used at the end of the section to obtain the convergence rate of Algorithm~\ref{alg:stoc_prox} in function values under convexity assumptions.

\begin{lem}\label{lem:key_lem}
	In general, for every index $t \geq 0$, we have
	\begin{equation}\label{eqn:recurs_1}
	\EE_t \|\hat x_t - x_{t+1}\|^2    \leq \|\hat x_t-x_{t}\|^2-\frac{\bar \rho-\tau-\eta}{\beta_t-\eta}\|\hat x_t-x_{t}\|^2  + \frac{4\lipsymb^2}{(\beta_t-\eta)(\beta_t-\bar \rho)}.
	\end{equation}
	Moreover, for any point $x\in \dom r$, the inequality holds:
	\begin{equation}\label{eqn:recurs_2}
	\EE_t\left[\| x_{t+1}-x\|^2  \right] \leq \frac{\beta_t+\tau}{\beta_t-\eta}\| x_{t}-x\|^2  - \frac{2}{\beta_t - \eta}\EE_t[\varphi(x_{t+1}) - \varphi(x)]+ \frac{2\lipsymb^2}{\beta_t(\beta_t-\eta)}.	
	\end{equation}
\end{lem}
\begin{proof}
	Recall that  the function $x\mapsto r(x) + f_{x_t}(x, \xi_t)+\frac{\beta_t}{2}\|x-x_t\|^2$ is strongly convex with constant $\beta_t-\eta$ and $x_{t+1}$ is its minimizer. Hence for any $x\in \dom r$, the inequality holds:
	\begin{align*}
	\left(r(x) + f_{x_t}(x, \xi_t)+ \tfrac{\beta_t}{2}\|x-x_t\|^2\right) \geq \Big(r(x_{t+1}) + &f_{x_t}( x_{t+1}, \xi_t) +\tfrac{\beta_t}{2} \|x_{t+1}-x_t\|^2\Big)\\
	&+ \tfrac{\beta_t-\eta}{2}\|x - x_{t+1}\|^2.
	\end{align*}
	Rearranging and taking expectations we successively deduce 
	\begin{align}
	&\EE_t\left[\frac{\beta_t - \eta}{2}\|x - x_{t+1}\|^2 +\frac{\beta_t}{2} \|x_{t+1}-x_t\|^2 - \frac{\beta_t}{2}\|x - x_{t}\|^2\right] \notag\\
	&\leq \EE_t[r(x) + f_{x_t}(x, \xi_t) -  r( x_{t+1}) - f_{x_t}( x_{t+1}, \xi_t) ]\notag\\
	&\leq \EE_t[r(x) +  f_{x_t}(x, \xi_t) -  r(x_{t+1}) -  f_{x_t}( x_{t}, \xi_t) +  L(\xi)\|x_{t+1}-x_t\|] \label{eqn:lip_modelcvx}\\
	&\leq r(x)+\EE_{\xi}[f_{x_t}(x,  \xi)]-\EE_{t}[r(x_{t+1})]-\EE_{\xi}[f_{x_t}( x_{t}, \xi)]\label{eqn:CS}\\
	&\quad+\sqrt{\EE_{\xi}[L(\xi)^2]}\cdot\sqrt{\EE_t[\|x_{t+1}-x_t\|^2]}\notag\\
	&\leq r(x)+f(x)-\EE_{t}[r(x_{t+1})]-f(x_t)+\frac{\tau}{2}\|x-x_t\|^2+\lipsymb\sqrt{\EE_t[\|x_{t+1}-x_t\|^2]}\label{eqn:apply_mod_expcvx}\\
	&= \EE_t[r(x)+f(x)-r(x_{t+1})-f(x_t)]+\frac{\tau}{2}\|x-x_t\|^2+\lipsymb\sqrt{\EE_t[\|x_{t+1}-x_t\|^2]}\notag\\
	&\leq \EE_t[r(x)+f(x)-r(x_{t+1})-f(x_{t+1})]+\frac{\tau}{2}\|x-x_t\|^2\label{eqn:apply_mod_exp2cvx}\\
	&\quad+\lipsymb\EE_t[\|x_{t+1}-x_t\|]+\lipsymb\sqrt{\EE_t[\|x_{t+1}-x_t\|^2]},\notag
\end{align}
where \eqref{eqn:lip_modelcvx} follows from Assumption (B4), inequality \eqref{eqn:CS} follows from Cauchy-Schwartz, inequality \eqref{eqn:apply_mod_expcvx} follows from (B2),  \eqref{eqn:apply_mod_exp2cvx} follows from  Lemma~\ref{lem:weak_conv}.

	Define $\delta:=\sqrt{\EE_t[\| x_{t+1}-x_t\|^2]}$ and notice $\delta\geq \EE_t\|x_t - x_{t+1}\|$. Rearranging \eqref{eqn:apply_mod_exp2cvx}, we immediately deduce
	\begin{align*}
	\EE_t\left[\frac{\beta_t - \eta}{2}\| x - x_{t+1}\|^2  \right] &\leq \EE_t\left[\frac{\beta_t+\tau}{2}\|x^\ast- x_{t}\|^2\right]  - \frac{\beta_t\delta^2}{2} +  2\lipsymb\delta - \EE_t[\varphi(x_{t+1}) -  \varphi(x)]\\
	&\leq \EE_t\left[\frac{\beta_t+\tau}{2}\|x- x_{t}\|^2\right]  + \frac{2\lipsymb^2}{\beta_t} - \EE_t[\varphi(x_{t+1}) -  \varphi(x)],	\end{align*}
	where the last inequality follows by maximizing the right-hand-side in $\delta \in\R$. Dividing through by $\frac{\beta-\eta}{2}$, we arrive at the claimed inequality \eqref{eqn:recurs_2}.

Next setting $x = \hat x_t$ in \eqref{eqn:apply_mod_exp2cvx} and using the definition of  the prox-point, we obtain
\begin{align}
&\EE_t\left[\frac{\beta_t - \eta}{2}\|\hat x_t - x_{t+1}\|^2 +\frac{\beta_t}{2} \|x_{t+1}-x_t\|^2 - \frac{\beta_t}{2}\|\hat x_t - x_{t}\|^2\right] \notag\\
	&\leq  \EE_t\left[-\frac{\bar \rho}{2}\|\hat x_t - x_{t}\|^2 + \frac{\bar\rho}{2}\|x_{t+1} - x_t\|^2\right]+\frac{\tau}{2}\|\hat x_t-x_t\|^2+2\lipsymb\delta\notag \\
	&= \frac{\tau-\bar \rho}{2}\|\hat x_t - x_{t}\|^2+\frac{\bar \rho}{2}\cdot\EE_t[\|x_{t+1}-x_t\|^2]+2\lipsymb\delta.\notag
	\end{align}
Rearranging, we deduce
	\begin{align}
	\frac{\beta_t - \eta}{2}\cdot\EE_t\|\hat x_t - x_{t+1}\|^2&\leq \frac{\beta_t-\bar\rho+\tau}{2}\|\hat x_t-x_t\|^2+\frac{\bar\rho-\beta_t}{2}\delta^2+2\lipsymb\delta\label{eqn:max_gamcvx}\\
	&\leq \frac{\beta_t-\bar\rho+\tau}{2}\|\hat x_t-x_t\|^2+\frac{2\lipsymb^2}{\beta_t-\bar\rho},\notag
	\end{align}
	where the last inequality follows by maximizing the right-hand-side of \eqref{eqn:max_gamcvx} in $\delta \in\R$. 
After multiplying through by $\frac{2}{\beta_t-\eta}$, we arrive at the claimed estimate \eqref{eqn:recurs_1}.
\end{proof}

We can now establish the convergence guarantees of Algorithm~\ref{alg:stoc_prox}.
\begin{thm}[Convergence rate]\label{thm:main_theorem}
	Fix a real $\bar \rho> \tau+\eta$ and a
	 sequence $\{\beta_t\}_{t\geq 0}\in (\bar \rho, \infty)$.
	Then the iterates $x_t$ generated by Algorithm~\ref{alg:stoc_prox} satisfy
\begin{equation}\label{eqn:desc_4}
\EE \left[\varphi_{1/\bar\rho}(x_{t+1})\right]\leq \EE[\varphi_{1/\bar\rho}(x_{t})]-\frac{\bar \rho - \tau - \eta}{2\bar \rho(\beta_t-\eta)}\EE\left[\|\nabla \varphi_{1/\bar\rho}(x_t)\|^2\right]  + \frac{2\bar\rho \lipsymb^2}{(\beta_t-\eta)(\beta_t-\bar \rho)},
\end{equation}
and the point $x_{t^*}$ returned by Algorithm~\ref{alg:subgradient} satisfies:
\begin{equation}\label{eqn:recurse_tower_4}
\EE\|\nabla \varphi_{1/\bar \rho}(x_{t^*})\|^2\leq \frac{\bar \rho(\varphi_{1/\bar \rho}(x_{0})-\min_x \varphi)+  2{\bar \rho}^2 \lipsymb^2\cdot\sum_{t=0}^T\frac{1}{(\beta_t-\eta)(\beta_t - \bar \rho)}}{\sum_{t=0}^T \frac{\bar \rho - \tau - \eta}{2(\beta_t-\eta)}}.
\end{equation} 
In particular, if Algorithm~\ref{alg:subgradient}  uses the constant parameter $\beta_t=\bar\rho+\gamma^{-1}\sqrt{T+1}$, for some real $\gamma>0$, then the point $x_{t^*}$ satisfies:
	\begin{equation}\label{eqn:compl_proj4}
	\EE\|\nabla \varphi_{1/\bar \rho}(x_{t^*})\|^2\leq\frac{2\left(\bar \rho(\varphi_{1/\bar \rho}(x_{0})-\min_x \varphi)+  2{\bar \rho}^2 \lipsymb^2\gamma^2\right)}{\bar \rho-\tau-\eta}\cdot\left(\frac{\bar\rho-\eta}{T+1}+\frac{1}{\gamma\sqrt{T+1} }\right).
	\end{equation}
\end{thm}
\begin{proof}
	Using the definition of the Moreau envelope and appealing to the estimate \eqref{eqn:recurs_1} in Lemma~\ref{lem:key_lem}, we deduce  
	\begin{align*}
	\EE_t[ \varphi_{1/\bar \rho}(x_{t+1})] &\leq \EE_t\left[ \varphi( \hat x_{t}) + \frac{\bar \rho}{2} \| x_{t+1} - \hat x_t \|^2 \right] \\
	&\leq  \varphi( \hat x_{t})  +\frac{\bar \rho}{2}\cdot \EE_t\left[\| x_{t+1} - \hat x_t \|^2\right],\\
	&\leq \varphi( \hat x_{t}) +\frac{\bar \rho}{2}\left[\|\hat x_t-x_{t}\|^2-\frac{\bar \rho-\tau-\eta}{\beta_t-\eta}\|\hat x_t-x_{t}\|^2  + \frac{4\lipsymb^2}{(\beta_t-\eta)(\beta_t-\bar \rho)}\right]\\
	&= \varphi_{1/\bar \rho}(x_{t}) - \frac{\bar \rho - \tau - \eta}{2\bar{\rho}(\beta_t-\eta)}\|\nabla \varphi_{1/\bar \rho}(x_t)\|^2  + \frac{2\bar\rho \lipsymb^2}{(\beta_t-\eta)(\beta_t-\bar \rho)}.
	\end{align*}
	 Taking expectations  with respect to $\xi_0,\ldots, \xi_{t-1}$ and using the tower rule  yields the claimed inequality \eqref{eqn:desc_4}.
		Unfolding the recursion \eqref{eqn:desc_4} yields: 
	$$\EE[ \varphi_{1/\bar \rho}(x_{t+1})]\leq \varphi_{1/\bar \rho}(x_{0})-\sum_{t=0}^T \left[\frac{\bar \rho - \tau - \eta}{2\bar{\rho}(\beta_t-\eta)}\EE[\|\nabla \varphi_{1/\bar \rho}(x_{t})\|^2]\right]  + 2\bar \rho \lipsymb^2\cdot\sum_{t=0}^T\tfrac{1}{(\beta_t-\eta)(\beta_t - \bar \rho)}.$$
	Using the inequality $\varphi_{1/\bar \rho}(x_{t+1})\geq \min \varphi$ and rearranging yields
	$$\sum_{t=0}^T \frac{\bar \rho - \tau - \eta}{\beta_t-\eta}\EE[\|\varphi_{1/\bar{\rho}}(x_t)\|^2]\leq 2\bar{\rho}(\varphi_{1/\bar \rho}(x_{0})-\min \varphi)+ 4 \lipsymb^2\bar \rho^2\sum_{t=0}^T\frac{1}{(\beta_t-\eta)(\beta_t - \bar \rho)}$$
	Dividing through by $\sum_{t=0}^T \tfrac{\bar \rho - \tau - \eta}{\beta_t-\eta}$ and recognizing the left side as
	$\EE[\|\varphi_{1/\bar{\rho}}(x_{t^*})\|^2]$ yields \eqref{eqn:recurse_tower_4}. Setting $\bar{\rho}=2\rho$ and $\beta_t=\bar\rho+\gamma^{-1}\sqrt{T+1}$ in~\eqref{eqn:recurse_tower_4} yields the final guarantee \eqref{eqn:compl_proj4}.
\end{proof}

Next we will look at the ``convex setting'', that is when the models $\EE_{\xi}f(\cdot,\xi)$ globally lower abound $f$, without quadratic error, and the functions $ f_x(\cdot,\xi)+r(\cdot)$ are $\mu$-strongly convex. By analogy with the stochastic subgradient method, one would expect that Algorithm~\ref{alg:stoc_prox} drives the  function gap $\EE[\varphi(x_t)-\varphi(\cdot))]$ to zero at the rates $O(\frac{1}{\sqrt{t}})$ and $O(\frac{1}{\mu t})$, in the settings $\mu =0$ and $\mu>0$, respectively. The following two theorems establish exactly that. Even when specializing to the stochastic proximal subgradient method, Theorems~\ref{thm:convergenceC_nonstrong} and \ref{thm:convergenceC} improve on the state of the art. In contrast to previous work \cite{prox_subgrad_duchi,cruz_subgrad},  the norms of the subgradients of $r$ do not enter the complexity bounds established in Theorem~\ref{thm:convergenceC_nonstrong}, while Theorem~\ref{thm:convergenceC}  extends the nonuniform averaging technique of \cite{simpler_subgrad} for strongly convex minimization to the fully proximal setting.

\begin{theorem}[Convergence rate under convexity] \label{thm:convergenceC_nonstrong}\hfill \\
Suppose that $\tau=0$ and the functions $f_x(\cdot,\xi)+r(\cdot)$ are convex. Let $\{x_t\}$ be the iterates generated by  Algorithm~\ref{alg:stoc_prox} and set $\alpha_t=\beta_t^{-1}$. Then for all $T > 0$, we have
\begin{equation}\label{eqn:conv_1}
\EE\left[\varphi\left(\tfrac{1}{\sum_{t=0}^T \alpha_t }\sum_{t=0}^T\alpha_t x_{t+1}\right) - \varphi(x^\ast)\right] \leq \frac{\tfrac{1}{2}\|x_0-x^\ast\|^2 + \lipsymb^2 \sum_{t=0}^T \alpha_t^2 }{\sum_{t=0}^T \alpha_t},
\end{equation}
where $x^*$ is any minimizer of $\varphi$. In particular,  if Algorithm~\ref{alg:subgradient}  uses the constant parameter $\alpha_t=\frac{\gamma}{\sqrt{T+1}}$, for some real $\gamma>0$, then the estimate holds
\begin{equation}\label{eqn:conv_3}
\EE\left[\varphi\left(\tfrac{1}{T+1 }\sum_{t=1}^{T+1} x_{t}\right) - \varphi(x^\ast)\right] \leq \frac{\tfrac{1}{2}\|x_0-x^\ast\|^2 + \lipsymb^2 \gamma^2 }{\gamma\sqrt{T+1}}.
\end{equation}

\end{theorem}
\begin{proof}
Setting $\eta:=0$ and $x:=x^*$ in the estimate \eqref{eqn:recurs_2} in Lemma~\ref{lem:key_lem}, and taking expectations of both sides yields 
$$2\alpha_t\EE[\varphi(x_{t+1}) - \varphi(x^*)]\leq \EE\| x_{t}-x^*\|^2  -\EE\left[\| x_{t+1}-x^*\|^2  \right]+ 2\lipsymb^2\alpha_t^2.$$
The  estimate \eqref{eqn:conv_1} then follows by summing across $t=0,\ldots, T$, dividing through by $\sum_{t=0}^T\alpha_t$, and using convexity of $\varphi$. The estimate \eqref{eqn:conv_3} is immediate from \eqref{eqn:conv_1}.
\end{proof}

The following theorem uses the nonuniform averaging technique from \cite{simpler_subgrad}.

\begin{theorem}[Convergence rate under strong convexity] \label{thm:convergenceC}
Suppose that $\tau=0$ and the functions $f_x(\cdot,\xi)+r(\cdot)$ are $\mu$-strongly convex for some $\mu>0$. Then for all $T > 0$, the iterates generated by  Algorithm~\ref{alg:stoc_prox} with $\beta_t=\frac{\mu(t+1)}{2}$ satisfy
\begin{align*}
\EE\left[\varphi\left(\tfrac{2}{(T+2)(T+3)-2}\sum_{t=1}^{T+1} (t+1)x_{t}\right) - \varphi(x^\ast) \right] \leq \frac{\mu\|x_0-x^*\|^2}{(T+2)^2}+\frac{8\lipsymb^2}{\mu(T+2)}.
\end{align*}
where $x^*$ is any minimizer of $\varphi$.
\end{theorem}

\begin{proof}
Define $\Delta_t:=\frac{1}{2}\EE[\|x-x_{t}\|^2]$. Setting $\eta:=-\mu$ and $x:=x^*$ in the estimate \eqref{eqn:recurs_2} of Lemma~\ref{lem:key_lem}, taking expectations of both sides, and multiplying through by $(\beta_t+\mu)/2$ yields 
$$\EE[\varphi(x_{t+1}) - \varphi(x^*)]\leq \beta_t\Delta_t  -(\beta_t+\mu){\Delta_{t+1}}+ \frac{\lipsymb^2}{\beta_t}.$$
Plugging in $\beta_t:=\frac{\mu(t+1)}{2}$, multiplying through by $t+2$, and summing, we get
\begin{align*}
\sum_{t=0}^T(t+2)\EE[\varphi(x_{t+1}) - \varphi(x^*)]&\leq \sum_{t=0}^T\left(\tfrac{\mu(t+1)(t+2)}{2}\Delta_t  -\tfrac{\mu(t+2)(t+3)}{2}{\Delta_{t+1}}\right)+ \sum_{t=0}^T\tfrac{2\lipsymb^2(t+2)}{\mu(t+1)}\\
&\leq \mu\Delta_0+\frac{4\lipsymb^2(T+1)}{\mu}
\end{align*}
Dividing through by the sum $\sum_{t=0}^T(t+2)=\frac{(T+2)(T+3)}{2} -1$ and using convexity of $\varphi$, we deduce
\begin{align}
\EE\left[\varphi\left(\tfrac{2}{(T+2)(T+3)-2}\sum_{t=0}^{T} (t+2)x_{t+1}\right) - \varphi(x^\ast) \right] &\leq \tfrac{\mu\|x_0-x^*\|^2}{(T+2)(T+3)-2}+ \tfrac{8\lipsymb^2(T+1)}{\mu((T+2)(T+3)-2)}\notag\\
&\leq \frac{\mu\|x_0-x^*\|^2}{(T+2)^2}+\frac{8\lipsymb^2}{\mu(T+2)},\label{eqn:last_ineq_bach}
\end{align}
where \eqref{eqn:last_ineq_bach} uses the estimate $(T+2)(T+3)-2\geq (T+2)^2$. The proof is complete.
\end{proof}

\subsection{Algorithmic examples}\label{sec:alg_examples}
Let us now look at the consequences of Theorem~\ref{thm:main_theorem} and Theorem~\ref{thm:convergenceC_nonstrong}. We begin with the algorithms  briefly mentioned in the introduction: stochastic proximal point, prox-linear, and proximal subgradient. In each case, we list the standard assumptions under which the methods are applicable, and then verify properties (B1)-(B4) for some $\tau,\eta,\lipsymb\geq 0$. Complexity guarantees for each method then follow immediately from Theorem~\ref{thm:main_theorem}. We then describe the problem of minimizing the expectation of a pointwise maximum of convex function (e.g. Conditional Value-at-Risk), and describe a natural model based algorithm for the problem. Convergence guarantees in function values then follow from Theorem~\ref{thm:convergenceC_nonstrong}.

\smallskip
\subsection*{Stochastic proximal point}
Consider the optimization problem~\eqref{eqn:stoch_approx} under the following assumptions.
\begin{enumerate}
	\item[(C1)] It is possible to generate i.i.d.\ realizations $\xi_1,\xi_2, \ldots \sim P$.
	\item[(C2)] There is an open convex set $U$ containing $\dom r$ and a measurable function $(x,y,\xi)\mapsto f_x(y,\xi)$ defined on $U\times U\times \Omega$ satisfying  $\EE_{\xi}[f_x(y,\xi)]=f(y)$ for all $x,y\in U$.
	\item[(C3)] Each function $r(\cdot)+f_x(\cdot,\xi)$ is $\rho$-weakly convex $\forall x\in U$, a.e. $\xi\in \Omega$.

	\item[(C4)] There exists a measurable function $L \colon \Omega\to \R_+$  satisfying 
	$\sqrt{\EE_{\xi}\left[L(\xi)^2\right]} \leq \lipsymb$
	and such that 
	$$f_x(x,\xi)-f_x(y,\xi)\leq L(\xi)\|x-y\|,$$ 
	 for all $x,y\in U$ and a.e. $\xi\in \Omega$.
\end{enumerate}
The stochastic proximal point method is 
Algorithm~\ref{alg:stoc_prox} with the models 
$f_{x}(y,\xi)$. It is immediate to see that (B1)-(B4) hold with $\tau=0$ and $\eta=\rho$.

\smallskip

\subsection*{Stochastic proximal subgradient}
We next slightly loosen the assumptions (A1)-(A3) for the proximal stochastic subgradient method, by allowing $r$ to be nonconvex, and show how these assumptions imply (B1)-(B4).
Consider the optimization problem~\eqref{eqn:stoch_approx}, and let us assume that the following properties are true.
\begin{enumerate}
	\item[(D1)] It is possible to generate i.i.d.\ realizations $\xi_1,\xi_2, \ldots \sim P$.
	\item[(D2)] The function $f$ is $\rho_1$-weakly convex and $r$ is $\rho_2$-weakly convex, for some $\rho_1,\rho_2\geq 0$.
	\item[(D3)] There is an open convex set $U$ containing $\dom r$ and a measurable mapping $G \colon U \times \Omega \rightarrow \RR^d$ satisfying  $\EE_{\xi}[G(x,\xi)]\in \partial f(x)$ for all $x\in U$.
	\item[(D4)] There is a real $\lipsymb\geq 0$ such that the inequality, $\EE_\xi\left[ \|G(x, \xi)\|^2\right] \leq \lipsymb^2$, holds for all $x \in U$. 
\end{enumerate}
The stochastic subgradient method is Algorithm~\ref{alg:stoc_prox} with the linear models 
$$f_{x}(y,\xi)=f(x)+\left\langle G(x,\xi),y-x\right\rangle.$$
Observe that (B1) and (B3) with $\eta=\rho_2$ are immediate from the definitions; 
(B2) with $\tau=\rho_1$ follows from the discussion in \cite[Section 2]{davis2018stochastic}. Assumption (B4) is also immediate from (D4).

\subsection*{Stochastic prox-linear}
Consider the optimization problem~\eqref{eqn:stoch_approx} with 
$$f(x)=\EE_{\xi\sim P}\left[h\big(c(x,\xi),\xi\big)\right].$$
We assume that there exists an open convex set $U$ containing $\dom r$ such that the following properties are true.
\begin{enumerate}
	\item[(E1)] It is possible to generate i.i.d.\ realizations $\xi_1,\xi_2, \ldots \sim P$.
	\item[(E2)] The assignments $h\colon \R^m\times\Omega\to \R$ and $c\colon U\times\Omega\to \R^m$ are measurable.	
	\item[(E3)] The function $r$ is $\rho$-weakly convex, and there exist square integrable functions $\ell,\gamma,M\colon\Omega\to\R$ such that  for a.e. $\xi\in \Omega$, the function  $z\mapsto h(z,\xi)$ is convex and $\ell(\xi)$-Lipschitz, the map $x\mapsto c(x,\xi)$ is $C^1$-smooth with $\gamma(\xi)$-Lipschitz Jacobian, and the inequality, 
	$\|\nabla c(x,\xi)\|_{{\rm op}}\leq M(\xi)$, holds for all $x\in U$ and a.e. $\xi\in \Omega$.
	\end{enumerate}

The stochastic prox-linear method \cite{duchi_ruan} is Algorithm~\ref{alg:stoc_prox} with the convex models 
$$f_{x}(y,\xi)=h\big(c(x,\xi)+\nabla c(x,\xi)(y-x),\xi\big).$$
Observe that (B1) and (B3) hold trivially with $\eta=\rho$. Assumption (B2) holds with $\tau=\sqrt{\EE_{\xi}[\ell(\xi)]^2} \sqrt{\EE_{\xi}[\gamma(\xi)^2]}$ by \cite[Lemma 3.12]{duchi_ruan}. Assumption (E3) also directly implies (B4) with $\lipsymb=\sqrt{\EE_{\xi}[\ell(\xi)]^2} \sqrt{\EE_{\xi}[M(\xi)^2]}$.

\subsection*{Expectation of convex monotone compositions}
As an application of Theorem~\ref{thm:convergenceC_nonstrong}, suppose  we wish to optimize the problem \eqref{eqn:stoch_approx}, where $r$ is convex and $f$ is given by 
$$f(x)=\EE_{\xi}[h(c(x,\xi),\xi)],$$
Suppose that $h(\cdot,\xi)\colon\R\to\R$ and $c(\cdot,\xi)\colon\R\to\R$ are convex, and $h(\cdot,\xi)$ is also nondecreasing. Note that we do not assume smoothness of $c(\cdot,x)$ and therefore this problem class does not fall with the composite framework discussed above. 

We assume that there exists an open convex set $U$ containing $\dom r$ such that the following properties are true.
\begin{enumerate}
	\item[(F1)] It is possible to generate i.i.d.\ realizations $\xi_1,\xi_2, \ldots \sim P$.
	\item[(F2)] The assignments $h\colon \R\times\Omega\to \R$ and $c\colon U\times\Omega\to \R^m$ are measurable, the functions $h(\cdot,\xi)$, $c(\cdot,\xi)$, and $r(\cdot)$ are convex, and $h(\cdot,\xi)$ is also nondecreasing.
	\item [(F3)] There is a measurable mapping $G \colon U \times \Omega \rightarrow \RR^d$ satisfying  $G(x,\xi)\in \partial_x c(x,\xi)$ for all $x\in U$.

	\item[(F4)] There exist square integrable functions $\ell, M\colon\Omega\to\R$ such that  for a.e. $\xi\in \Omega$, the function  $z\mapsto h(z,\xi)$ is $\ell(\xi)$-Lipschitz and the map $x\mapsto c(x,\xi)$ is $M(\xi)$-Lipschitz for a.e. $\xi\in \Omega$.
	\end{enumerate}
	
One reasonable class of models then reads:
$$f_{x}(y,\xi)=h(c(x,\xi)+\langle G(x,\xi),y-x\rangle,\xi).$$ 
Assumption (B1) is immediate from (F1). Assumption $(F2)$ directly implies $(B2)$ with $\tau=0$ and (B3) with $\eta=0$. Finally (F4) readily implies (B2) with $\lipsymb=\sqrt{\EE_{\xi}[\ell(\xi)]^2} \sqrt{\EE_{\xi}[M(\xi)^2]}$. Thus the stochastic model-based algorithm (Algorithm~\ref{alg:stoc_prox}) enjoys the $O(\frac{1}{\sqrt{t}})$ convergence guarantee in expected function value gap (Theorem~\ref{thm:convergenceC_nonstrong}).

As an illustration, consider the Conditional Value-at-Risk problem, discussed in Example~\ref{exa:cVAr}:
$$\min_{\gamma\in \R, x\in \R^d}~ (1-\alpha)\gamma+\EE_{\xi\sim P}[(g(x,\xi)-\gamma)^+] +r(x),$$
under the assumption that the loss $g(\cdot,\xi)$ is convex. Then given an iterate $(x_t,\gamma_t)$, the stochastic model based algorithm would sample $\xi_t\sim P$, choose a subgradient $v_t\in \partial_x g(x_t,\xi_t)$ and perform the simple update
\begin{align*}
(x_{t+1},\gamma_{t+1})=\argmin_{\gamma\in \R, y\in \R^d}~ (1-\alpha)\gamma+&\left[g(x_t,\xi_t)+\langle v_t,y-x_t \rangle -\gamma\right]^+ +r(y)\\
&\qquad+\frac{\beta_t}{2}(\|y-x_t\|^2+\|\gamma-\gamma_t\|^2).
\end{align*}

\section{Numerical Illustrations}\label{sec:num_ill} 
In this section, we illustrate our three running examples (stochastic subgradient, prox-linear, prox-point) on phase retrieval and blind deconvolution problems, outlined in Section~\ref{eqn:comp}. In particular, our experiments complement the recent paper \cite{duchi_ruan}, which performs an extensive numerical study of the stochastic subgradient and prox-linear algorithms on the phase retrieval problem.

 Our main goal in this section is to illustrate that the update rules for all three algorithms, have essentially the same computational cost. Indeed, the subproblems for the stochastic prox-point and prox-linear algorithms have a closed form solution. Note that our theoretical guarantees (Theorem~\ref{thm:main_theorem}) imply essentially the same worst-case complexity for the stochastic subgradient, prox-linear, and proximal point algorithms. In contrast, our numerical results on both problems clearly show that the latter two algorithms are much better empirically both in terms of speed and robustness to the choice of stepsize. 
 
\subsection{Phase retrieval} 
The experimental set-up for the phase retrieval problem is as follows. We  generate standard Gaussian measurements $a_i\sim N(0,I_{d\times d})$,  for $i=1,\ldots, m$; generate the target signal $\bar x$ and initial point $x_0$ uniformly on the unit sphere; and set $b_i = \langle a_i, \bar x\rangle^2$ for each $i=1,\ldots, m$. We then apply the three stochastic algorithms to the problem 
$$\min_{x\in \R^d}~ \frac{1}{m}\sum_{i=1}^m |\langle a_i,x\rangle^2-b_i|.$$
Each step of the algorithms is trivial to implement. Since the three methods only use one data point at a time, let us define the function 
$$g(x)=|\langle a,x\rangle^2-b|,$$
for a fixed vector $a\in\R^d$ and a real $b\geq 0$.

\paragraph{Stochastic subgradient}
The stochastic subgradient method simply needs to evaluate an element of the subdifferential
$$\partial g(x)=2\langle a,x\rangle a\cdot \left\{\begin{array}{lr}
        \sign(\langle a,x\rangle^2-b), & \text{if } \langle a,x\rangle^2\neq b\\
        ~[-1,1], & \text{o.w.}
        \end{array}\right\}.$$

\paragraph{Stochastic prox-linear}
The stochastic prox-linear method needs to solve subproblems of the form
$$\argmin_{\Delta\in\R^d}~|\langle a,x\rangle^2+2\langle a,x\rangle\langle a,\Delta\rangle-b|+\frac{1}{2\lambda}\|\Delta\|^2.$$
Then the next iterate is defined to be $x+\Delta$. Setting $\gamma=\lambda(\langle a,x\rangle^2-b)$ and $\zeta=2\lambda\langle a,x\rangle a$, we therefore seek to solve the problem 
\begin{equation}\label{eqn:uni_var}
\argmin_{\Delta\in\R^d}~|\gamma+\langle \zeta,\Delta\rangle|+\frac{1}{2}\|\Delta\|^2.
\end{equation}
An explicit solution $\Delta^*$ to this subproblem follows from a standard Lagrangian calculation, and is recorded for example in \cite[Section 4]{duchi_ruan}:
\begin{equation} \label{eqn:uni_var_soln}
\Delta^*=\proj_{[-1,1]}\left(\frac{-\gamma}{\|\zeta\|^2}\right) \zeta.
\end{equation}

\paragraph{Stochastic proximal point}
Finally, the stochastic proximal point method requires solving the problem 
\begin{equation}\label{eqn:subproblem_prox_phase}
\argmin_{y}~|\langle a,y\rangle^2-b|+\frac{1}{2\lambda}\|y-x\|^2
\end{equation}
Let us compute the candidate solutions using first-order optimality conditions:
\begin{equation} \label{eqn:optimal_cond}
\lambda^{-1}(x-y)\in 2\langle a,y\rangle a\cdot \left\{\begin{array}{lr}
        \sign(\langle a,y\rangle^2-b), & \text{if } \langle a,y\rangle^2\neq b\\
        ~[-1,1], & \text{o.w.}
        \end{array}\right\}.
        \end{equation}
An easy computation shows that there are at most four point $y$ that satisfy \eqref{eqn:optimal_cond}:
$$\left\{x-\left(\frac{2\lambda\langle a,x\rangle}{2\lambda \|a\|^2\pm1}\right) a,~x-\left(\frac{\langle a,x\rangle\pm\sqrt{b}}{\|a\|^2}\right) a\right\}.$$
Therefore we may set the next iterate to be the candidate solution $y$ with the lowest function value for the subproblem \eqref{eqn:subproblem_prox_phase}.

We perform three sets of experiments corresponding to $(d,m)=(10,30)$,  $(50,150)$, $(100,300)$, and record the result in Figure~\ref{fig:num_ill_phase}. The dashed blue line indicates the initial functional error. In each set of experiments, we use 100 equally spaced step-size parameters $\beta_t^{-1}$ between $10^{-4}$ and $1$. The figures on the left record the function gap after 100 passes through the data, averaged over 15 rounds. The figures on the  right output the number of epochs used by the stochastic prox-linear and proximal point methods to find a point achieving $10^{-4}$ functional suboptimality, averaged over 15 rounds. It is clear from the figures that the stochastic prox-linear and proximal point algorithms perform much better and are more robust to the choice of the step-size parameter than the stochastic subgradient method. 

\begin{figure}[h!]
\begin{subfigure}{.5\textwidth}
  \centering
   \includegraphics[scale=0.47]{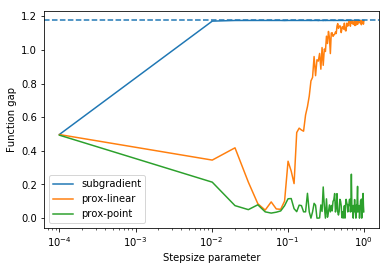} 
\end{subfigure}%
\begin{subfigure}{.5\textwidth}
  \centering
   \includegraphics[scale=0.47]{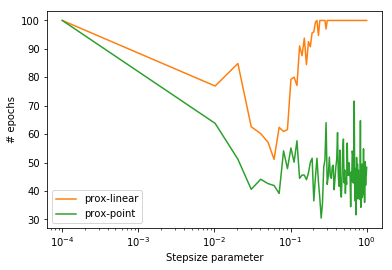}  
\end{subfigure}

\begin{subfigure}{.5\textwidth}
  \centering
   \includegraphics[scale=0.47]{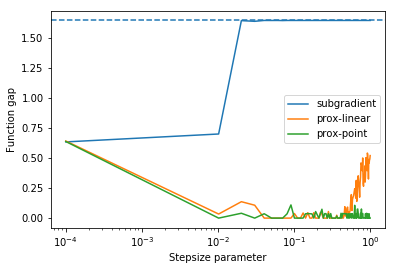}
\end{subfigure}%
\begin{subfigure}{.5\textwidth}
  \centering
  \includegraphics[scale=0.47]{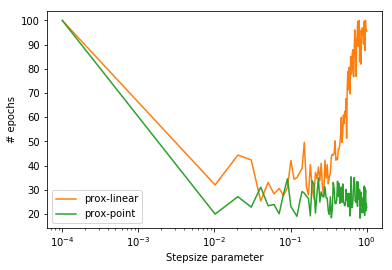}
\end{subfigure}

\begin{subfigure}{.5\textwidth}
  \centering
  \includegraphics[scale=0.47]{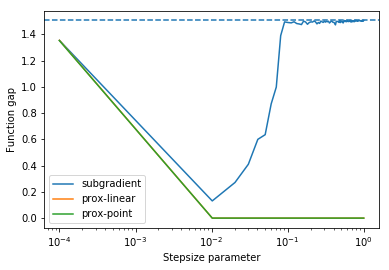}
\end{subfigure}%
\begin{subfigure}{.5\textwidth}
  \centering
  \includegraphics[scale=0.47]{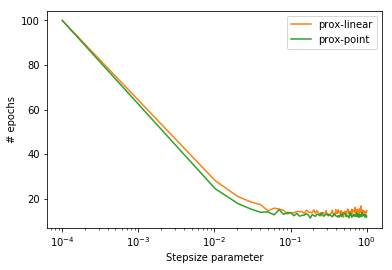}
\end{subfigure}
  \caption{Bottom to top: $(d,m)=(10,30),  (50,150),(100,300)$. The dashed blue line indicates the initial functional error.}
  \label{fig:num_ill_phase}
\end{figure}

\subsection{Blind deconvolution}
We next consider the problem of blind deconvolution. The experimental set-up is as follows. We  generate standard Gaussian measurements $u_i\sim N(0,I_{d_1\times d_1})$ and $v_i\sim N(0,I_{d_2\times d_2})$,  for $i=1,\ldots, m$; generate the target signal $\bar x$ uniformly on the unit sphere; and set $b_i = \langle u_i, \bar x\rangle\langle v_i, \bar x\rangle$ for each $i=1,\ldots, m$. The problem formulation reads:
$$\min_{x,y}~ \frac{1}{m}\sum_{i=1}^m | \langle u_i,x\rangle \langle v_i, y\rangle-b_i|,$$
Again, since the three methods access one data point at a time, define the function 
$$g(x,y)=| \langle u,x\rangle \langle v, y\rangle-b|$$
for some vectors $u\in \R^d_1$ and $v\in \R^d_2$ and real $b\in \R$.

\paragraph{Stochastic subgradient}
The stochastic subgradient method, in each iteration,  evaluates an element of the subdifferential
$$\partial g(x,y)=\left(\langle v,y\rangle u,\langle u,x\rangle v\right)\cdot \left\{\begin{array}{lr}
\sign(\langle u,x\rangle \langle v, y\rangle-b), & \text{if } \langle u,x\rangle \langle v, y\rangle\neq b\\
~[-1,1], & \text{o.w.}
\end{array}\right\}.$$

\paragraph{Stochastic prox-linear}
The stochastic prox-linear method needs to solve subproblems of the form:
$$\argmin_{\Delta_1,\Delta_2} ~| \langle u, x\rangle \langle v,  y\rangle+\langle v,y\rangle \langle u ,\Delta_1\rangle+\langle u,x\rangle\langle v,\Delta_2\rangle -b|+\frac{1}{2\lambda}(\|\Delta_1\|^2+\| \Delta_2\|^2).$$
Once a solution $(\Delta_1,\Delta_2)$ is found, the next iterate is $(x+\Delta_1,y+\Delta_2)$. Clearly, we may rewrite the prox-linear subproblem in the form \eqref{eqn:uni_var} under the identification $\Delta=(\Delta_1,\Delta_2)$, $\zeta=\lambda(\langle v,y\rangle u,\langle u,x\rangle v)$, and $\gamma=\lambda(\langle u, x\rangle \langle v,  y\rangle-b)$. We may then read off the solution directly from \eqref{eqn:uni_var_soln}.

\paragraph{Stochastic proximal point}
Finally, the stochastic proximal point method requires solving the problem 
\begin{equation}\label{eqn:proximal_point_deconv}
\argmin_{x,y}~| \langle u,x\rangle \langle v, y\rangle-b|+\frac{1}{2\lambda}\|x-x_0\|^2+\frac{1}{2\lambda}\|y-y_0\|^2.
\end{equation}
Let us enumerate the critical points. Writing out the optimality conditions for $(x,y)$, there are two cases to consider. In the first case $\langle u,x\rangle \langle v, y\rangle\neq b$, it is straightforward to show that the possible critical point have the form
\begin{equation}\label{eqn:step_1_expr_blind}
\begin{aligned}
x&=x_0-\lambda\left(\frac{\pm\langle v,y_0\rangle-\lambda\|v\|^2\langle u,x_0\rangle}{1-\lambda^2\|u\|^2\|v\|^2}\right) u\\
y&=y_0-\lambda\left(\frac{\pm\langle u,x_0\rangle-\lambda\|u\|^2\langle v,y_0\rangle}{1-\lambda^2\|u\|^2\|v\|^2}\right) v
\end{aligned}
\end{equation}
Indeed, suppose for the moment $\langle u,x\rangle \langle v, y\rangle>b$. Then optimality conditions for \eqref{eqn:proximal_point_deconv} imply
\begin{equation}\label{eqn:step_update_blind}
\begin{aligned}
x&=x_0-\lambda\langle v,y\rangle u,\qquad 
y=y_0-\lambda\langle u,x\rangle v
\end{aligned}
\end{equation}
Thus if we determine $\langle v,y\rangle$ and $\langle u,x\rangle$, we will have an explicit formula for $(x,y)$.
Taking the dot product of the first equation with $u$ and the second with $v$ yields
\begin{equation*}
\lambda\langle v,y\rangle \|u\|^2=\langle u,x_0\rangle-\langle u,x\rangle,\qquad \lambda\langle u,x\rangle \|v\|^2=\langle v,y_0\rangle-\langle v,y\rangle.
\end{equation*}
Solving for $\langle v,y\rangle$ and $\langle u,x\rangle$, we get
\begin{equation*}
\langle u,x\rangle=\frac{\langle u,x_0\rangle-\lambda\|u\|^2\langle v,y_0\rangle}{1-\lambda^2\|u\|^2\|v\|^2},\qquad
\langle v,y\rangle=\frac{\langle v,y_0\rangle-\lambda\|v\|^2\langle u,x_0\rangle}{1-\lambda^2\|u\|^2\|v\|^2}.
\end{equation*}
Combining  these expressions with \eqref{eqn:step_update_blind}, we deduce that $x$ and $y$ can be expressed as in \eqref{eqn:step_1_expr_blind}. The setting  $\langle u,x\rangle \langle v, y\rangle>b$ is completely analogous.

In the second case, suppose $\langle u,x\rangle \langle v, y\rangle=b$. Then optimality condition for \eqref{eqn:proximal_point_deconv} imply that there exists $\gamma$ such that 
\begin{equation*}
x=x_0-\gamma\langle v,y\rangle u,\qquad y=y_0-\gamma\langle u,x\rangle v,\qquad b=\langle u,x\rangle\langle v,y\rangle.
\end{equation*}
We must solve this system of equations for $\gamma$, $\eta:=\langle u,x\rangle$, and $\delta:=\langle v,y\rangle$. Substituting the third equation into the first yields:
\begin{equation}\label{eqn:weird_asswtf}
x=x_0-\gamma\left(\frac{b}{\eta}\right) u,\qquad
y=y_0-\gamma\eta v,\qquad
b=\eta\delta.
\end{equation}
Taking the dot product of the first equation with $u$ and the second with $v$ yields
\begin{equation*}
\eta=\langle u,x_0\rangle-\gamma\left(\frac{b}{\eta}\right) \|u\|^2,\qquad 
\frac{b}{\eta}=\langle v,y_0\rangle-\gamma\eta \|v\|^2.
\end{equation*}
Solving the first equation for $\gamma$, we get  the expression
$\gamma=\frac{\eta\langle u,x_0\rangle-\eta^2}{b\|u\|^2}$. Plugging this formula into the second equation and clearing the denominator, we arrive at the quartic polynomial
$$0=\eta^4\|v\|^2-\eta^3\|v\|^2\langle u,x_0\rangle+b\eta\|u\|^2\langle v,y_0\rangle-b^2\|u\|^2.$$
Thus after finding each root $\eta$, we may set 
$\gamma=\frac{\eta\langle u,x_0\rangle-\eta^2}{b\|u\|^2}$, and then obtain an explicit formula for $(x,y)$ using  \eqref{eqn:weird_asswtf}.

Our numerical experiments are similar to those for phase retrieval. We perform three sets of experiments corresponding to $(d_1,d_2,m)=(10,10,30)$, $(50,50,200)$, $(100,100,400)$, and record the result in Figure~\ref{fig:num_blind}. The dashed blue line indicates the initial functional error. In each set of experiments, we use 100 equally spaced step-size parameters $\beta_t^{-1}$ between $10^{-4}$ and $1$. The figures on the left record the function gap after 100 passes through the data, averaged over 10 rounds. The figures on the  right output the number of epochs used by the stochastic prox-linear and proximal point methods to find a point achieving $10^{-4}$ functional suboptimality, averaged over 10 rounds. As in phase retrieval, it is clear from the figure that the stochastic prox-linear and proximal point algorithms perform much better and are more robust to the choice of the step-size parameter than the stochastic subgradient method. 


\begin{figure}[h!]
\begin{subfigure}{.5\textwidth}
  \centering
   \includegraphics[scale=0.47]{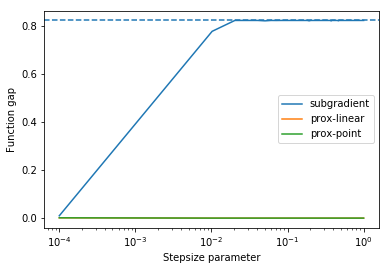} 
\end{subfigure}%
\begin{subfigure}{.5\textwidth}
  \centering
   \includegraphics[scale=0.47]{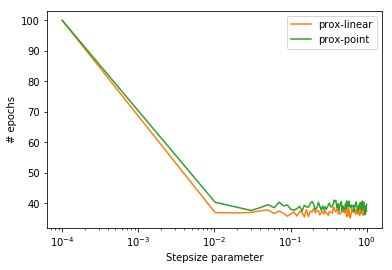}  
\end{subfigure}\\

\begin{subfigure}{.5\textwidth}
  \centering
   \includegraphics[scale=0.47]{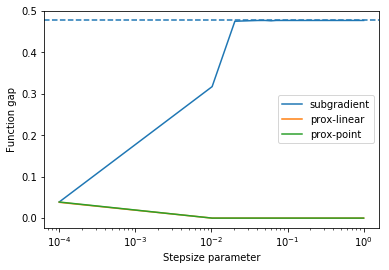} 
\end{subfigure}%
\begin{subfigure}{.5\textwidth}
  \centering
   \includegraphics[scale=0.47]{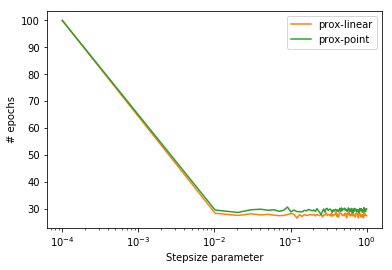}  
\end{subfigure}\\

\begin{subfigure}{.5\textwidth}
  \centering
   \includegraphics[scale=0.47]{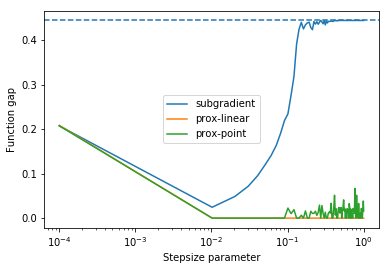} 
\end{subfigure}%
\begin{subfigure}{.5\textwidth}
  \centering
   \includegraphics[scale=0.47]{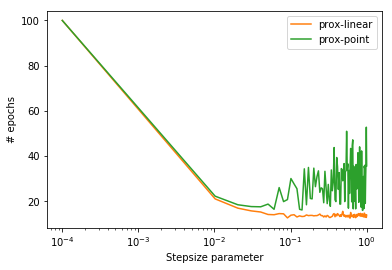}  
\end{subfigure}

  \caption{Bottom to top: $(d_1,d_2,m)=(10,10,50),(50,50,200), (100,100,400)$. The dashed blue line indicates the initial functional error.}
  \label{fig:num_blind}
\end{figure}

\section*{Acknowledgments}
The authors thank John Duchi for his careful reading of the initial version of this manuscript.

\bibliographystyle{siamplain}
\bibliography{bibliography}
\end{document}